\documentclass[12pt]{amsart}
\usepackage{calligra}
\usepackage{amsthm}
\usepackage{amsmath}
\usepackage{amssymb}
\usepackage{amscd}
\usepackage{bbm}
\usepackage[all]{xy}
\usepackage{mathrsfs}
\usepackage[top=72pt,bottom=96pt,left=72pt,right=72pt]{geometry}
\newtheorem{Lemma}{Lemma}[section]
\newtheorem{Remark}[Lemma]{Remark}

\newtheorem{Theorem}[Lemma]{Theorem}

\newtheorem{Proposition}[Lemma]{Proposition}
\numberwithin{equation}{section}

\def\qS{\mathscr{S}}
\def\C{\mathbb{C}}

\def\Hor{\mathrm{Hor}}
\def\dvol{\mathrm{dvol}}
\def\id{\mathrm{id}}
\def\Ker{\mathrm{Ker}}
\def\Im{\mathrm{Im}}
\def\YM{\mathrm{YM}}

\def\inv{\mathrm{inv}}
\def\H{\mathrm{Hor}}
\def\Mor{\textsc{Mor}}

\def\R{\mathbb{R}}
\def\Z{\mathbb{Z}}
\def\C{\mathbb{C}}

\def\l{\mathrm{L}}

\def\ad{\mathrm{ad}}
\def\Ad{\mathrm{Ad}}

\def\r{\mathrm{R}}

\begin{document}
\date{\today}
\title{Classical Electromagnetic Field Theory with Quantum Principal Bundles}
\author{Gustavo Amilcar Salda\~na Moncada}
\address{Gustavo Amilcar Salda\~na Moncada\\
CIMAT, Unidad Guanajuato}
\email{gamilcar@ciencias.unam.mx}
\begin{abstract}
The aim of this paper is to formulate a {\it non--commutative geometrical} version of the classical electromagnetic field theory in the vacuum with the Moyal--Weyl algebra as the space--time by using the theory of quantum principal bundles and quantum principal connections. As a result we will present the correct Maxwell equations in the vacuum of the model, in which we can appreciate the existence of electric and magnetic charges and currents. Finally, in the fourth section we are going to present a {\it mathematical model} for which there are instantons that are not solutions of the corresponding Yang--Mills equation. 

 \begin{center}
  \parbox{300pt}{\textit{MSC 2010:}\ 46L87, 58B99.}
  \\[5pt]
  \parbox{300pt}{\textit{Keywords:}\ Quantum principal bundles, quantum principal connections, Maxwell equations, instantons.}
 \end{center}
\end{abstract}
\maketitle
\section{Introduction}
It is well--known the relationship between geometry and physics; particularly when we deal with the electromagnetic theory \cite{na}. Indeed, one of the most general starting points of this theory in the vacuum is to consider it as a Yang--Mills theory for the trivial principal $U(1)$--bundle over the Minkowski space--time: $\R^4$ with the metric $\eta=\mathrm{diag}(1,-1,-1,-1)$ \cite{na}. 

In Appendix A we will show a brief summary of the classical electromagnetic field theory in the vacuum. In particular, we comment how the (second) Bianchi identity (see equation \ref{1.2}) gives rise to the Gauss Law for magnetism (equation \ref{1.7}) and the Faraday equation (see equation \ref{1.8}); while  critical points of the Yang--Mills functional (see equation \ref{1.3}), the functional that measures the square norm of the curvature of a principal connection, gives rise to the Gauss Law (equation \ref{1.9}) and the Ampere equation (equation \ref{1.10}).

It is worth remarking that in equation \ref{1.3} (or equivalently the equations \ref{1.9}, \ref{1.10}) we are looking for critical points of the Yang--Mills functional, so in principle, not every principal connection (or equivalently, not every $1$--form potential) can satisfy it. The equation \ref{1.3} is usually called the {\it dynamical equation} because it rules the dynamics of the electric and the magnetic field.

On the other side, equation \ref{1.2} (or equivalently the equations \ref{1.7}, \ref{1.8}) is always satisfied for every potential $1$--form (or equivalently, for every principal connection). This equation can be thought as a condition imposed on the electromagnetic field by the Minkowski space--time and the group $U(1)$. So the equation \ref{1.2} is usually called the {\it geometrical (or topological) equation}, reflecting the fact that it comes from the Bianchi equation (and the topology of the bundle). 

This paper aims to show two goals. The first one is using the general theory presented in \cite{sald1,sald2} to recreate the geometrical formulation of the classical electromagnetic field theory showed in the Appendix A by using quantum principal bundles and quantum principal connections in a concrete non--commutative space--time: the Moyal--Weyl algebra \cite{05}. Of course, in the text we will show the link between our formulation with the common results about the electromagnetic theory in non--commutative geometry, for example in \cite{u(1),elec,gauge2,twisted,ruiz}.  In particular, we are going to present the correct geometrical (topological) equation by using the non--commutative Bianchi identity, and the correct dynamical equation by identifying critical points of the non--commutative Yang--Mills functional, which is now the functional that measures the square norm of the curvature of a quantum principal connection (qpc). With that we will present the correct Maxwell equations in the vacuum as well as the Lagrangian of the system. An interesting result of this formulation is the fact that the four--potential produces by self--interaction {\it electric charges and currents densities in the vacuum and  magnetic charges and currents densities in the vacuum too}. As far as we know, as of now these charge/current terms in the vacuum are not taken into account in the literature \cite{u(1),elec,gauge2,twisted,ruiz}.

The second goal of this paper consists in using the general theory presented in \cite{sald1,sald2} to create a {\it mathematical model} of the classical electromagnetic field theory in the Moyal--Weyl algebra such that there exist instantons that are not solutions of the corresponding Yang--Mills equation. To create this model we retake the Cabibbo--Ferrari's idea (\cite{cf,cv}), which consists in having two gauge fields associated to the electromagnetic interaction: an {\it electric photon field} and an  {\it magnetic photon field} in order to have full symmetry between the electric and the magnetic field. Nevertheless, in the {\it classical} case the existence of two gauge fields leads to a theory with $U(1)\times U(1)$ as gauge group (\cite{cf}); while in our model the gauge group is still $U(1)$. This is only a {\it mathematical model} in the sense that it does not represent the physical word because there are not two gauge fields associated to the electromagnetic interaction. However this model opens the possibility to find more realistic models for which there are instantons that are not solutions of the Yang--Mills equation.

The importance of this paper lies, not only in its results in the Moyal--Weyl algebra (which is one of the most studied spaces for quantum gravity in the framework of non--commutative geometry), but also in the generality of the formulation because with that we have been able to recreate previous results and go further by explaining the mathematical reasons of the terms in the theory. This should be not surprising given the relationship between principal bundles and gauge theory in the {\it classical} case.  It is worth mentioning that it is possible to use the general theory of \cite{sald1,sald2} to create non--commutative gauge theories for other compact Lie groups, like $SU(2)$ or $SU(3)$. 

This paper breaks down into five sections. Following this introduction, in the second section we are going to build the quantum principal $U(1)$--bundle used. In the third section we will present the non--commutative Maxwell equations for our model by only using the geometry of the spaces and in this section we will show the covariant formulation of our model, as well as some non--trivial solutions. In addition, in the fourth section we are going to present the model for which not every instanton is a solution of the corresponding Yang--Mills equation. The last section is for some concluding comments and just like we have mentioned before, in Appendix A there is a brief summary of the geometrical formulation of the electromagnetic field theory in the vacuum in the {\it classical} case.

 To accomplish our purposes, we are going to use M. Durdevich's theory of quantum principal bundles and quantum principal connection (\cite{micho1,micho2,micho3,steve}) and the general formulation of the Yang--Mills theory presented in \cite{sald1,sald2} which has been tested in other quantum bundles with several exciting and interesting results, as in \cite{sald3,sald4}. It is worth remarking that the theory showed in \cite{sald1,sald2} was formulated in the most general way. Thus it was not created for the particular case of the quantum bundle used in this paper. 

 We have chosen Durdevich's framework to develop this paper instead of the framework presented in \cite{libro,brz,qvbH} because of its purely geometrical/algebraic formulation and because of its generality in terms of differential calculus.  

In Non--Commutative Geometry it is common to use the word {\it quantum} as synonymous of {\it non--commutative} and we will use it in the paper sometimes. On the other hand, in physics it is common to use the word {\it non--commutative} to denote gauge theories with non--abelian groups, which is not the case of this paper. So we expect that the reader does not confuse with all these terms.

\section{The Trivial Quantum Principal $U(1)$--Bundle}

In this section we are going to present the quantum principal bundle in which we will work. Since in the {\it classical} case it is used a trivial principal bundle to develop the electromagnetic field theory in the vacuum and we are interested only in changing the geometry to a non-commutative one (we are not interested in changing the {\it topology} of the bundle), we have to consider a trivial quantum principal $U(1)$--bundle. The general theory of this kind of quantum bundles can be checked in \cite{micho2,steve}.

\subsection{A Non--Commutative Minkowski Space--Time}

Let us start by considering the Minkowski space--time $(\R^4,\eta=\mathrm{diag}(1,-1,-1,-1))$ and its space of complex--valued smooth functions $C^\infty_\C(\R^4)$. By choosing a $4\times 4$ antisymmetric matrix $(\theta^{\mu\nu})$ $\in$ $M_4(\R)$, it is possible to take $C^\infty_\C(\R^4)[[\theta^{\mu\nu}]]$ the formal power series in $\theta^{\mu\nu}$ with coefficients in $C^\infty_\C(\R^4)$. Finally it is possible to apply a $\theta^{\mu\nu}$--twist on $C^\infty_\C(\R^4)[[\theta^{\mu\nu}]]$ by defining a new product. For every $f$, $h$ $\in$ $C^\infty_\C(\R^4)[[\theta^{\mu\nu}]]$ we define

\begin{equation}
    \label{2.1}
    f\cdot h:=m\circ \mathrm{exp}\left( \frac{i\,\theta^{\mu\nu}}{2} \frac{\partial}{\partial x^\mu}\otimes \frac{\partial}{\partial x^\nu} \right)(f\otimes h),
\end{equation}
where $m$ denotes the usual product on $C^\infty_\C(\R^4)[[\theta^{\mu\nu}]]$ and we have used Einstein summation convention with $i=\sqrt{-1}$. Explicitly we have
$$f\cdot h:= \left( fh +  \frac{i\,\theta^{\mu\nu}}{2} \frac{\partial f}{\partial x^\mu} \frac{\partial h}{\partial x^\nu}+ \frac{i^2\,\theta^{\mu\nu}\, \theta^{\alpha\beta}}{8} \frac{\partial f}{\partial x^\alpha \partial x^\mu} \frac{\partial h}{\partial x^\beta \partial x^\nu}+\cdots \right).$$  With this new product, $C^\infty_\C(\R^4)[[\theta^{\mu\nu}]]$ is a  non--commutative unital $\ast$--algebra (\cite{05}), where the unit is $\mathbbm{1}(x)=1$ for all $x$ and the $\ast$ operation is the complex conjugate.

It is worth mentioning that 
\begin{equation}
    \label{2.2}
    [x^\mu,x^\nu]:=x^\mu\cdot x^\nu-x^\nu \cdot x^\mu =i\,\theta^{\mu\nu}\,\mathbbm{1}.
\end{equation}
This non--commutative unital $\ast$--algebra is given the name of Moyal--Weyl algebra (\cite{05}) and it will be our quantum space--time, which we are going to denote by $B$. 

The next step is to extend the Moyal product $\cdot$ to the space of differential forms \cite{05}. In fact, let us take $\Omega^\bullet_\C(\R^4)[[\theta^{\mu\nu}]]$ the space of formal power series in $\theta^{\mu\nu}$ with coefficient in the algebra of complex--valued differential forms. Now equation \ref{2.1} is easily extended to $\Omega^\bullet_\C(\R^4)[[\theta^{\mu\nu}]]$ by considering the action of $ \frac{\partial}{\partial x^\gamma}$ on forms by means of the Lie derivative.  We will be denoted by $\Omega^\bullet(B)$ the space $\Omega^\bullet_\C(\R^4)[[\theta^{\mu\nu}]]$ with the Moyal product $\cdot$ extended. This graded differential $\ast$--algebra is going to play the role of {\it quantum differential forms} in  $B$.

In accordance with \cite{sald2}, in order to apply its theory there are certain structures that we have to define on $\Omega^\bullet(B)$.

Let us define the following left quantum Pseudo--Riemmanian metric on $B$
\begin{equation}
\label{2.3}
\{\langle-,-\rangle^k_\l: \Omega^k(B)\times \Omega^k(B) \longrightarrow B\mid k=0,1,2,3,4\}
\end{equation}
such that $\langle f ,h\rangle^0_\l= f\cdot h^\ast$ and for $k=1,2,3,4$ it is the usual metric of the de Rham differential algebra of the Minkowski space--time. For example $$\displaystyle \langle \sum^3_{\mu=0}f_\mu\,dx^\mu ,\sum^3_{\nu=0}h_\nu\,dx^\nu\rangle^1_\l= \sum^3_{\mu,\nu=0}\eta^{\mu\,\nu} f_\mu\cdot h^\ast_\nu=f_0\cdot h^\ast_0-f_1\cdot h^\ast_1-f_2\cdot h^\ast_2-f_3\cdot h^\ast_3$$ and $$\langle f\,\dvol,h\,\dvol\rangle^4_\l=f\cdot h^\ast,$$ where 
\begin{equation}
    \label{2.4}
    \dvol:=dx^0\wedge dx^1\wedge dx^2\wedge dx^3
\end{equation}
is the volume form. Furthermore, by postulating the orthogonality between forms of different degrees, we can induce a Pseudo--Riemannian structure in the whole graded space. So we will not use superscripts anymore. 

By taking the usual integral operator on the Moyal--Weyl algebra  $$\int_{B}\dvol$$ we can define the left quantum Hodge pseudo inner product
\begin{equation}
    \label{2.5}
     \langle-|-\rangle_\l:= \int_{B} \langle-,-\rangle_\l \,\dvol.
\end{equation}
Furthermore and according to \cite{sald2}, the left quantum Hodge operator is the antilinear $B$--isomorphism 
\begin{equation}
    \label{2.6}
     \star_\l:\Omega^k(B)\longrightarrow \Omega^{4-k}(B)
\end{equation}
that satisfies $$\hat{\mu}\cdot (\star_\l \mu) =\langle \hat{\mu},\mu\rangle_\l \,\dvol.$$ Explicitly, for the canonical basis we get
\begin{equation}
    \label{2.7}
    \star_\l\mathbbm{1}=\dvol, \;\;\star_\l \dvol=-\mathbbm{1},
\end{equation}
\begin{equation}
    \label{2.8}
    \begin{aligned}
    \star_\l dx^0=dx^1\wedge dx^2\wedge dx^3, \quad \star_\l dx^1=dx^0\wedge dx^2\wedge dx^3,\\
    \star_\l dx^2=-dx^0\wedge dx^1\wedge dx^3,\quad \star_\l dx^3=dx^0\wedge dx^1\wedge dx^2,
    \end{aligned}
\end{equation}
\begin{equation}
    \label{2.9}
    \begin{aligned}
    \star_\l dx^0 \wedge dx^1=-dx^2\wedge dx^3, \quad \star_\l dx^0 \wedge dx^2=dx^1\wedge dx^3, \quad \star_\l dx^0 \wedge dx^3=-dx^1\wedge dx^2,\\
    \star_\l dx^1 \wedge dx^2=dx^0\wedge dx^3, \quad \star_\l dx^1 \wedge dx^3=-dx^0\wedge dx^2, \quad \star_\l dx^2 \wedge dx^3=dx^0\wedge dx^1,
    \end{aligned}
\end{equation}
\begin{equation}
    \label{2.10}
    \begin{aligned}
    \star_\l dx^1 \wedge dx^2\wedge dx^3 =dx^0, \quad \star_\l dx^0 \wedge dx^2 \wedge dx^3=dx^1,\\
    \star_\l dx^0 \wedge dx^1 \wedge dx^3=-dx^2, \quad \star_\l dx^0 \wedge dx^1 \wedge dx^2=dx^3.
    \end{aligned}
\end{equation}
Here $$d:\Omega^k(B)\longrightarrow \Omega^{k+1}(B) $$ is the differential operator and of course, in all cases we have $\star^2_\l=(-1)^{k(4-k)+1}\,\id.$  Finally the left quantum codifferential is defined as the operator
\begin{equation}
    \label{2.11}
    d^{\star_\l}:=(-1)^{k+1}\,\star^{-1}_\l\,\circ\, d \,\circ\, \star_\l:\Omega^{k+1}(B)\longrightarrow \Omega^k(B),
\end{equation}
and it is the formal adjoint operator of $d$ with respect to $\langle-|-\rangle_\l$ \cite{sald2}.

It is worth mentioning that by considering the right structure as 
\begin{equation}
    \label{2.12}
    \langle \hat{\mu},\mu\rangle_\r:=\langle \hat{\mu}^\ast,\mu^\ast\rangle_\l
\end{equation}
we get $\langle-|-\rangle_\r$, $\star_\r:=\ast \circ \star_\l \circ \ast$ and $d^{\star_\r}:=(-1)^{k+1}\,\star^{-1}_\r\,\circ\, d \,\circ\, \star_\r=\ast \circ d^{\star_\l} \circ \ast$, which is the formal adjoint operator of $d$ with respect to $\langle-|-\rangle_\r$ \cite{sald2}.

\subsection{The Quantum Group of $U(1)$ and its Differential calculus}

Let us start this subsection by considering the $\ast$--Hopf algebra of the polynomial Laurent algebra $$(H=\C[z,z^{-1}],\Delta,\epsilon,S),$$ where $z^{-1}=z^\ast$, $\Delta$ is the coproduct, $\epsilon$ is the counit and $S$ the coinverse also called the antipode. These operations are given by $$\Delta(z^n)=z^n\otimes z^n,\qquad \epsilon(z^n)=1,\qquad S(z^n)=z^{n\ast}.$$ The space $H$ will play the role of the quantum structure group of our non--commutative bundle.

The next step is to find a differential calculus on $H$ different from the well--known algebra of {\it classical} differential forms. The reasons of this change will be discussed in the Remark \ref{rema2}. In accordance with \cite{steve,woro}, a bicovariant $\ast$--First Order Differential calculus ($\ast$--FODC) can be defined by an $\Ad$--invariant right ideal $\mathcal{R} \subseteq \Ker(\epsilon)$ such that $S(\mathcal{R})^\ast\subseteq \mathcal{R}$, where $\Ad:H\longrightarrow H\otimes H$ is the adjoint right coaction on $H$. In this way, let us consider any $\ast$--FODC
\begin{equation}
\label{2.13}
    (\Gamma,d)
\end{equation}
such that the set of invariant elements, or the {\it quantum (dual) Lie algebra} ${_\inv}\Gamma$ satisfies
$$\mathrm{dim}_\C({_\inv}\Gamma)=1,\qquad z\,\pi(z)=-\pi(z)\,z,$$ where $$\pi:H\longrightarrow {_\inv}\Gamma$$ is the quantum germs map given by $$\pi(g)=\kappa(g^{(1)})dg^{(2)}$$ for all $g$ $\in$ $H$ with $\Delta(g)=g^{(1)}\otimes g^{(2)}$ (in Sweedler's notation) \cite{steve,woro}. Notice 
\begin{equation}
\label{2.14}
    {_\inv}\Gamma=\mathrm{span}_\C\{\vartheta:=i\, \pi(z)\}.
\end{equation}
Also we can calculate the adjoint coaction of $H$ on ${_{\inv}}\Gamma$
\begin{equation}
    \label{2.15}
    \begin{aligned}
\ad: {_{\inv}}\Gamma &\longrightarrow {_{\inv}}\Gamma\otimes H\\
\vartheta &\longmapsto \vartheta\otimes \mathbbm{1}
\end{aligned}
\end{equation}
because of $$\ad(\pi(g))=((\pi\otimes \id)\circ \Ad)(g)=\pi(g)\otimes \mathbbm{1}$$ for all $g$ $\in$ $H$ \cite{micho1,micho2,steve}. 

In Durdevich's framework of quantum principal bundles we have to take the universal differential envelope $\ast$--calculus 
\cite{micho1,micho2,steve}
\begin{equation}
\label{2.16}
    (\Gamma^\wedge,d,\ast).
\end{equation}

This graded differential $\ast$--algebra will play the role of {\it quantum differential forms} of $U(1)$ and it has the particularity that, for example $$\Gamma^{\wedge\,0}=H,\qquad \Gamma^{\wedge\,1}=\Gamma=H\otimes {_\inv}\Gamma,\qquad \Gamma^{\wedge\,2}\cong H\otimes{_{\inv}}\Gamma^{\wedge\,2}$$ where 
\begin{equation}
\label{2.17}
    {_{\inv}}\Gamma^{\wedge\,2}=\mathrm{span}_\C\{ \vartheta\,\vartheta\}.
\end{equation}
So there are quantum differential forms of grade $2$. Moreover, there is not a top grade, which is a big difference from the algebra of {\it classical} differential forms of $U(1)$.

The main reason to use the universal differential envelope $\ast$--calculus instead of, for example, the universal differential calculus is the fact that $(\Gamma^\wedge,d,\ast)$ allows to extend the structure of $\ast$--Hopf algebra to any grade and it is maximal with this property (although sometimes both differential calculus agree); reflecting the {\it classical} fact the tangent bundle of a Lie group is also a Lie Group. The structure of graded differential $\ast$--Hopf algebra will be denoted by the same symbols.

In this specific case we have (\cite{micho1,micho2,steve})
\begin{equation}
    \label{2.18}
    \vartheta^\ast=\vartheta,\qquad d\vartheta=i\vartheta\,\vartheta.
\end{equation}

\subsection{The Quantum Bundle}

Finally we have all the ingredients to build the trivial quantum bundle that we will use in the rest of this section and the next one. Let
\begin{equation}
    \label{2.29} 
    \zeta=(P:=B\otimes H, B, \Delta_P:=\id\otimes \Delta)
\end{equation}
be the trivial quantum principal $U(1)$--bundle over $B$ \cite{micho2,steve}. Now we can also take the trivial differential calculus on the bundle (\cite{micho2,steve}):
\begin{equation}
    \label{2.30}
    \Omega^\bullet(P):=\Omega^\bullet(B)\otimes \Gamma^\wedge, \;\;\Delta_{\Omega^\bullet(P)}:=\id\otimes \Delta:\Omega^\bullet(P)\longrightarrow \Omega^\bullet(P)\otimes \Gamma^\wedge
\end{equation}
(here $\Delta$ is the extension of the coproduct in $\Gamma^\wedge$). The graded differential $\ast$--algebra $\Omega^\bullet(P)$ will play the role of {\it quantum differential forms} on the total quantum space $P$. It is worth mentioning that 
\begin{equation}
    \label{2.31}
    \Hor^\bullet P:=\Omega^\bullet(B)\otimes H
\end{equation}
and of course the graded differential $\ast$--subalgebra of forms on the base quantum space matches with $\Omega^\bullet(B)$. In accordance with the general theory, the restriction of $\Delta_{\Omega^\bullet(P)}$ in $ \Hor^\bullet P$ is a corepresentation of $H$ on $ \Hor^\bullet P$ \cite{micho2,steve}. This map will be denoted by $$\Delta_\H: \Hor^\bullet P\longrightarrow \Hor^\bullet P\otimes H$$

In light of \cite{micho2,steve}, the set of quantum principal connections (qpcs)
\begin{equation}
    \label{2.32}
    \{\omega: {_{\inv}}\Gamma\longrightarrow \Omega^1(P)\mid \omega \mbox{ is linear and }\Delta_{\Omega^\bullet(P)}(\omega(\theta))=(\omega\otimes \id_H)\circ \ad(\theta)+\mathbbm{1}\otimes \theta \mbox{ for all }\theta\}
\end{equation}
on trivial quantum bundles is in bijection with the space of {\it non--commutative gauge potentials}
\begin{equation}
    \label{2.33}
    \{A^\omega: {_{\inv}}\Gamma\longrightarrow \Omega^1(B)\mid A^\omega  \mbox{ is linear}  \}
\end{equation}
by means of $$\omega(\theta)=A^\omega(\theta)\otimes \mathbbm{1}+\mathbbm{1}\otimes \theta$$ for all $\theta$ $\in$ ${_{\inv}}\Gamma$. For $A^\omega=0$ the corresponding qpc is called the trivial qpc. 

Now we are going to talk about the curvature of a qpc. In Durdevich framework  it is necessary to choose a embedded differential (\cite{micho2,steve}) $$\Theta:{_{\inv}}\Gamma\longrightarrow {_{\inv}}\Gamma\otimes {_{\inv}}\Gamma.$$ Such maps have to satisfy $$M\circ (\ad\otimes\ad) \circ \Theta=(\Theta\otimes \id_H)\circ \Theta,\quad \Theta(\theta)=\theta^{(1)}\otimes \theta^{(2)}\quad \mbox{ and }  \quad \Theta(\theta^\ast)=-\theta^{(2)\,\ast}\otimes\theta^{(1)\,\ast}$$ if $d\theta=\theta^{(1)}\theta^{(2)}$, where $M(\theta_1,g_1,\theta_2,g_2)=(\theta_1,\theta_2,g_1g_2)$ for every $\theta_1$, $\theta_2$ $\in$ ${_{\inv}}\Gamma$ and every $g_1$, $g_2$ $\in$ $H$. In other words, we 
choose  $\Theta$ in a way compatible with  the differential structure for embedding ${_\inv}\Gamma$ into ${_\inv}\Gamma\otimes {_\inv}\Gamma$ (\cite{sald2}).

In our case, by the equations \ref{2.14}, \ref{2.18} there is {\it only one} embedded differential 
\begin{equation}
    \label{2.34}
    \Theta: {_{\inv}}\Gamma\longrightarrow {_{\inv}}\Gamma\otimes {_{\inv}}\Gamma
\end{equation}
which is given by $$\Theta(\vartheta)=i\,\vartheta\otimes \vartheta.$$ With this the space of curvatures of qpcs
\begin{equation}
    \label{2.35}
    \{R^\omega: {_{\inv}}\Gamma\longrightarrow \Omega^2(P) \mid R^\omega:=d\omega-\langle \omega,\omega\rangle \}
\end{equation}
is in bijection with the space of {\it non--commutative fields strength}
\begin{equation}
    \label{2.36}
    \{F^\omega:=dA^\omega-\langle A^\omega,A^\omega \rangle: {_{\inv}}\Gamma\longrightarrow \Omega^2(B)\}
\end{equation}
by means of (\cite{micho2,steve}) $$R^\omega(\theta)=(F^\omega\otimes \id_H)\circ \ad(\theta)$$ for all $\theta$ $\in$ ${_\inv}\Gamma$. Here we are considering that for a given algebra $X$ $$\langle C,D \rangle:=m\circ (C\otimes D)\circ \Theta$$ with $m$ the product map and $C,D:{_\inv}\Gamma \longrightarrow X$ are linear maps. The reader can compare this relation between qpcs and their curvatures in terms of potentials $$A^\omega \longleftrightarrow F^\omega=dA^\omega-\langle A^\omega,A^\omega \rangle$$ with its classical counterpart $$A^\omega \longleftrightarrow F^\omega=dA^\omega.$$ Of course, in general, the curvature of a qpc depends on the choice of $\Theta$. However, this is not our case since there is only one embedded differential. It is worth mentioning that the embedded differential map $\Theta$ in the definition of the curvature of a qpc allows seeing it as an element of $$\Mor(\ad, \Delta_\H)=\{ \tau:{_\inv}\Gamma \longrightarrow \Hor^\bullet P \mid \tau \mbox{ is a linear map such that } (\tau\otimes \id_H)\circ\ad=\Delta_\H \circ \tau\}$$ in the general case. \cite{micho2,steve}.

Following the {\it classical} case, let us consider the non--commutative gauge potential $A^\omega$ given by
\begin{equation}
    \label{2.37}
    A^\omega(\vartheta)=\phi\, dx^0-A_1\,dx^1-A_2\,dx^2-A_3\,dx^3.
\end{equation}
In this way, the non--commutative field strength is defined by
\begin{equation}
    \label{2.38}
    F^\omega(\vartheta)=dA^\omega(\vartheta)-i A^\omega(\vartheta)\wedge A^\omega(\vartheta)
\end{equation}
and in terms of coordinates we have 
\begin{equation}
    \label{2.39}
    F^\omega(\vartheta)=\sum_{0\leq \mu < \nu \leq 3}F_{\mu\nu}\, dx^\mu\wedge dx^\nu\;\; \mathrm{ where }\;\; (F_{\mu\nu})=\begin{pmatrix}
0 & D_1 & D_2 & D_3  \\
-D_1 & 0 & -H_3 & H_2 \\
-D_2 & H_3 & 0 & -H_1\\
-D_3 & -H_2 & H_1 & 0
\end{pmatrix},
\end{equation}
where
\begin{equation}
    \label{2.40}
    {\bf D}=(D_1,D_2,D_3):={\bf E}+i [\phi,{\bf A}],\qquad {\bf E}:=-{\partial {\bf A}\over \partial x^0}-\nabla \phi
\end{equation}
and
\begin{equation}
    \label{2.41}
    {\bf H}=(H_1,H_2,H_3):={\bf B}+ i {\bf A}\times {\bf A},\qquad {\bf B}:=\nabla \times {\bf A},
\end{equation}
with ${\bf A}=(A_1,A_2,A_3)$. The definition of the commutators can be deduced from the context and we have considered that in the cross product the multiplication of elements always starts from top to bottom.

The field ${\bf D}$, ${\bf H}$ will be considered as the non--commutative electric field and the non--commutative magnetic field, respectively. Also $F_{\mu\nu}$ will be called the non--commutative electromagnetic tensor field. It is worth noticing that ${\bf D}$, ${\bf H}$ have a {\it classical} part (${\bf E}$, ${\bf B}$) and a {\it quantum} part ($i [\phi,{\bf A}]$, $i {\bf A}\times {\bf A}$) which comes from $\Theta\not=0$ in the definition of the curvature.

\begin{Proposition}
    \label{cov}
    The covariant derivative $D^\omega$ of an element of $\Mor(\ad, \Delta_\H)$ is just the differential $d$ of the element, for every qpc $\omega$.
\end{Proposition}

\begin{proof}
    According to \cite{micho2,steve}, the covariant derivative of a qpc $\omega$ is the first--order linear map $$D^\omega: \Hor^\bullet P\longrightarrow \Hor^\bullet P $$ given by $D^\omega\psi=d\psi-(-1)^k\psi^{(0)}\omega(\pi(\psi^{(1)}))$ for all $\psi$ $\in$ $\Hor^k P$ with $\Delta_P(\psi)= \psi^{(0)}\otimes\psi^{(1)}$.
    
    Let $\tau$ $\in$ $\Mor(\ad,\Delta_\H)$ such that $\Im(\tau)\subseteq \Hor^k P $. Since $\mathrm{dim}({_\inv}\Gamma)=1$, $\tau$ is completely defined by its value on $\vartheta$; so $$ D^\omega \tau(\vartheta)=d\tau(\vartheta)-(-1)^k \tau(\vartheta)^{(0)}\omega(\pi(\tau(\vartheta)^{(1)}))=d\tau(\vartheta)$$ because  $(\tau\otimes \id_H)\circ\ad=\Delta_\H \circ \tau$ and $\pi(\mathbbm{1})=0$. Hence $D^\omega \tau=d\tau.$ 
\end{proof}

\begin{Remark}
    \label{rema}
    The last proposition agrees with its {\it classical} counterpart: for the trivial bundle $\mathrm{proj}:\R^4\times U(1)\longrightarrow \R^4$, the covariant derivative of every basic form $\tau$ of type $\ad$ (the adjoint action on the Lie algebra $\mathfrak{u}(1)$) is just $d\tau$. 
\end{Remark}

In accordance with  \cite{micho2,steve}, a regular qpc is a qpc for which its covariant derivative satisfies the graded Leibniz rule.   
\begin{Proposition}
\label{prop1}
The trivial qpc is the only regular connection.
\end{Proposition}
\begin{proof}
Let us assume that $\omega$ is a regular qpc. Then according to \cite{micho2,steve} its corresponding non--commutative gauge potential has to satisfies $$A^\omega(\theta \circ g)=\epsilon(g)A^\omega(\theta)$$ for all $\theta$ $\in$ ${_{\inv}}\Gamma $ and all $g$ $\in$ $H$, where $$\theta \circ g= \pi(hg-\epsilon(h)g)$$ if $\theta=\pi(h)$ for some $h$ $\in$ $H$. So it is enough to evaluate in $g=z$ to find that this condition is satisfied if and only if $A^\omega=0$.
\end{proof}

Despite the ease of the previous proof, the last proposition is quite important because every {\it classical} principal connection is regular \cite{micho2}. Thus the last result tells that except by the trivial qpc, there is no {\it classical} counterpart of any qpc, i.e., except by the trivial qpc (for which $F^\omega=0$), all results in the rest of this paper will be completely {\it quantum}; they will not have  {\it classical} analogues.

\section{Non--Commutative Electromagnetic Field Theory}

As we have exposed at the beginning of this paper, Maxwell equations in the vacuum come from the Bianchi identity and critical points of the Yang--Mills functional (see Appendix A). In this sections we are going to recreate that process in order to find their {\it non--commutative} counterparts in our quantum bundle. Furthermore, in this section we will present the covariant formulation of our development.

\subsection{Non--Commutative Geometrical (Topological) equation}

In accordance with \cite{micho2,sald2}, every qpc satisfies the {\it non--commutative Bianchi identity}, which is
\begin{equation}
    \label{3.1}
    (D^\omega-S^\omega)R^\omega=\langle \omega, \langle \omega,\omega\rangle\rangle-\langle \langle \omega,\omega\rangle,\omega\rangle,
\end{equation}
The definition of the operator $S^\omega$ will be given in the proof of Proposition \ref{prop1}. The operator $S^\omega$ measures the {\it lack of regularity} of the qpc $\omega$ in the sense that $S^\omega=0$ when $\omega$ is regular. Furthermore, when $\omega$ is multiplicative (\cite{micho2,steve}), the right--hand side of the last equation is equal to $0$. In summary, when $\omega$ is regular and multiplicative, for example, for {\it classical} principal connections, equation \ref{3.1} turns into the well--known {\it classical} Bianchi identity $$D^\omega R^\omega=0.$$ Of course in general $S^\omega \not=0$, $\langle \omega, \langle \omega,\omega\rangle\rangle-\langle \langle \omega,\omega\rangle,\omega\rangle\not=0$ and this makes the {\it quantum} case so interesting.

 For our calculus we have
\begin{Proposition}
    \label{prop1}
    The equation \ref{3.1} in terms of the non--commutative field strength $F^\omega$ is
\begin{equation}
    \label{3.2}
    (d-d^{S^\omega})F^\omega=0,
\end{equation}
where
\begin{equation}
    \label{3.3}
    d^{S^\omega} T: {_{\inv}}\Gamma\longrightarrow \Omega^{k+1}(B)
\end{equation}
is given by $$ d^{S^\omega}T(\vartheta)=i[A^\omega(\vartheta),T(\vartheta)]^\partial:= i(A^\omega(\vartheta)\wedge T(\vartheta)-(-1)^k\, T(\vartheta)\wedge A^\omega(\vartheta))\;\;\;\in\;\;\; \Omega^{k+1}(B)$$ for all linear maps $T: {_{\inv}}\Gamma\longrightarrow \Omega^k(B).$
\end{Proposition}

\begin{proof}
    The general definition of the operator $S^\omega$ is given by (\cite{micho2,sald2})
    $$S^\omega\circ \tau:=\langle \omega,\tau\rangle -(-1)^k \langle \tau, \omega\rangle -(-1)^k[\tau,\omega]:  {_{\inv}}\Gamma\longrightarrow \Hor^{k+1} P,$$ where $$[\tau,\omega]:=m\circ (\tau\otimes\omega)\circ (\id_{{_\inv}\Gamma}\otimes \pi)\circ \ad$$ and $\tau$ $\in$ $\Mor(\ad,\Delta_\H)$ such that $\Im(\tau)$ $\subseteq$ $\Hor^k P$. By equation \ref{2.15}, $[\tau,\omega]=0$ since $ \pi(\mathbbm{1})=0$ and hence $$S^\omega\circ \tau= \langle \omega,\tau\rangle -(-1)^k \langle \tau, \omega\rangle.$$ According to equation \ref{2.15}, $\tau(\vartheta)=T(\vartheta) \otimes \mathbbm{1}$ for some linear map $T:{_\inv}\Gamma \longrightarrow  \Omega^k(B)$ (and this induces a bijection between $\Mor(\ad,\Delta_\H)$ and linear maps $T:{_\inv}\Gamma\longrightarrow \Omega^\bullet(B)$ \cite{micho2,steve}). In this way we have $$(S^\omega\circ \tau )(\vartheta)= i(A^\omega(\vartheta)\wedge T(\vartheta)-(-1)^k T(\vartheta)\wedge A^\omega(\vartheta))\otimes \mathbbm{1}=i[A^\omega(\vartheta),T(\vartheta)]^\partial \otimes \mathbbm{1}\;\in \; \Omega^{k+1}(B)\otimes \mathbbm{1}.$$ Therefore, the linear map $ S^\omega\circ \tau$ in terms of $T$ 
    $$d^{S^\omega}T:{_\inv}\Gamma \longrightarrow \Omega^{k+1}(B) $$ is determined by
    $$ d^{S^\omega}T(\vartheta)=i[A^\omega(\vartheta),T(\vartheta)]^\partial.$$ Notice that the linear map $T$ associated to the curvature $R^\omega$ is the non--commutative field strength $F^\omega:{_\inv}\Gamma \longrightarrow \Omega^2(B)$.

    On the other hand, by Proposition \ref{cov}, $$D^\omega R^\omega=dR^\omega \qquad \mbox{ and } \qquad dR^\omega(\vartheta)=dF^\omega(\vartheta)\otimes \mathbbm{1}.$$ Thus in terms of $ F^\omega$, the covariant derivative of the curvature is just $dF^\omega$, as in the {\it classical} case.

    Finally, $$(\langle \omega, \langle \omega,\omega\rangle\rangle-\langle \langle \omega,\omega\rangle,\omega\rangle) (\vartheta)=(A^\omega(\vartheta)\cdot A^\omega(\vartheta)\cdot A^\omega(\vartheta)-A^\omega(\vartheta)\cdot A^\omega(\vartheta)\cdot A^\omega(\vartheta))\otimes \mathbbm{1}=0$$ and hence,  equation \ref{3.1} for the non--commutative field strength is the equation \ref{3.2}.
\end{proof}

 Evaluating  equation \ref{3.2} in the element $\vartheta$ we get 
\begin{equation}
    \label{3.4}
    dF^\omega(\vartheta)=i[A^\omega(\vartheta),dF^\omega(\vartheta)] =i [A^\omega(\vartheta),dA^\omega(\vartheta)] \;\in\; \Omega^3(B)
\end{equation}
The reader can compare the last equation with equation \ref{1.2}. 

\begin{Theorem}
    \label{max1}
    The non--commutative Gauss law for magnetism and the non--commutative Faraday equation are 
    \begin{equation}
    \label{3.5}
    \nabla\cdot {\bf H}=\rho^m,
\end{equation}
\begin{equation}
    \label{3.6}
     \nabla\times {\bf D}+\frac{\partial {\bf H}}{\partial x^0}=-{\bf j}^m,
\end{equation}
where 
\begin{equation}
    \label{3.7}
     \rho^m:=i[{\bf B},{\bf A}]
\end{equation}
and
\begin{equation}
    \label{3.8}
     -{\bf j}^m:=i[\phi,{\bf B}]-i({\bf E}\times {\bf A}+{\bf A}\times {\bf E})
\end{equation}
\end{Theorem}

\begin{proof}
    As in the {\it classical} case, a direct calculation in coordinates shows that  $$dF^\omega(\vartheta)=-\nabla\cdot {\bf H}\,dx^1\wedge dx^2 \wedge dx^3-\left(\nabla\times {\bf D}+\frac{\partial {\bf H}}{\partial x^0}\right)_1 \,dx^0\wedge dx^2 \wedge dx^3$$ $$+\left(\nabla\times {\bf D}+\frac{\partial {\bf H}}{\partial x^0}\right)_2 \,dx^0\wedge dx^1 \wedge dx^3-\left(\nabla\times {\bf D}+\frac{\partial {\bf H}}{\partial x^0}\right)_3 \,dx^0\wedge dx^1 \wedge dx^2,$$ where $\displaystyle \left(\nabla\times {\bf D}+\frac{\partial {\bf H}}{\partial x^0}\right)_i$ denotes the $i$--coordinate of the vector $\displaystyle \nabla\times {\bf D}+\frac{\partial {\bf H}}{\partial x^0}$. While  $$d^{S^\omega}F^\omega(\vartheta)=-\rho^m \,dx^1\wedge dx^2 \wedge dx^3 + {\bf j}^m_1 \, dx^0\wedge dx^2 \wedge dx^3 -{\bf j}^m_2 \,dx^0\wedge dx^1 \wedge dx^3+{\bf j}^m_3 \,dx^0\wedge dx^1 \wedge dx^2,$$ where $$\rho^m=i[{\bf B},{\bf A}] \qquad \mbox{ and } \qquad -{\bf j}^m=-({\bf j}^m_1,{\bf j}^m_2,{\bf j}^m_3)=i[\phi,{\bf B}]-i({\bf E}\times {\bf A}+{\bf A}\times {\bf E}).$$ Now the theorem follows from equation \ref{3.4} and the linear independence of the canonical basis $\{dx^\mu\wedge dx^\nu \wedge dx^\sigma \mid 0 \leq \mu < \nu < \sigma \leq 3\}$
\end{proof}

\begin{Remark}
    \label{remax1}
    Following the {\it classical} interpretation, since the term $\rho^m$ is equal to the divergence of the field ${\bf H}$, we can consider that $\rho^m$ is a (non--commutative) magnetic charge density. In the same way, since $-{\bf j}^m$ is equal to the rotational of the field ${\bf D}$ plus the partial of time of the field ${\bf H}$, we can consider that $-{\bf j}^m$ is a (non--commutative) magnetic current density. It is highly worth remembering that we are considering only an electromagnetic field in the vacuum.
\end{Remark}

Again, the definition of the commutators can be deduced from the context. The reader can compare the equations \ref{3.5}, \ref{3.6} with their {\it classical counterparts} in equations \ref{1.7}, \ref{1.8}. Notice that the magnetic charge and current depend on the interaction of $A^\omega(\vartheta)$ with the {\it classical} part of the non--commutative field strength $F^\omega(\vartheta)$ $\in$ $\Omega^2(B)$.

\begin{Remark}
    \label{remax2}
    As in the {\it classical} case, equation \ref{3.4} or equivalently equations \ref{3.5}, \ref{3.6} are satisfied by all non--commutative gauge potential $A^\omega$. In Subsection 3.3 we are going to present an explicit $A^\omega$ for which $\rho^m\not=0$.
    
    Thus if we ask for $$dF^\omega(\vartheta)=0$$ or equivalently $$\nabla\cdot {\bf H}=0, \qquad \nabla\times {\bf D}+\frac{\partial {\bf H}}{\partial x^0}=0$$ in an attempt to imitate the {\it classic} case as in \cite{obs}, that would be just a special case of equation \ref{3.4} and hence not every $A^\omega$ would satisfy those conditions. In other words, $dF^\omega(\vartheta)=0$ cannot be the correct geometrical (topological) equation of our trivial quantum bundle. The right consequence of the Bianchi identity in non--commutative geometry is the equation  \ref{3.4}.
\end{Remark}

\subsection{Non--Commutative Dynamical equation} 

In accordance with \cite{sald2}, a qpc $\omega$ is a Yang--Mills qpc, i.e., a critical point of the non--commutative Yang--Mills functional (the functional that measures the square norm of the curvature of a qpc) if and only if
\begin{equation}
\label{3.18}
\langle \Upsilon_\ad\circ \lambda\,|\,(d^{\nabla^{\omega}_{\ad}\star_\l}-d^{S^{\omega}\star_\l}) R^{\omega}\rangle_\l+\langle \widetilde{\Upsilon}_\ad\circ \widehat{\lambda}\,|\,(d^{\widehat{\nabla}^{\omega}_{\ad}\star_\r}-d^{\widehat{S}^{\omega}\star_\r}) \widehat{R}^{\omega}\rangle_\r=0
\end{equation} 
for all $\lambda$ $\in$ $\Mor^1(\ad, \Delta_\H):=\{ \tau:{_\inv}\Gamma \longrightarrow \Hor^1 P \mid \tau \mbox{ is a linear map such that }(\tau\otimes \id_H)\circ\ad=\Delta_\H \circ \tau\}$. Here $\Upsilon_\ad$ is the canonical isomorphism between $$\Mor(\ad,\Delta_\Hor) \quad \mbox{ and }\quad \Omega^\bullet(B)\otimes_B E^\ad_\l,$$ where $E^\ad_\l:=\Mor(\ad,\Delta_P)$ as left $B$--module is the left associated quantum vector bundle of $\ad$. Moreover $$\langle-\mid -\rangle_\l$$ is the induced (pseudo) inner product defined on $\Omega^\bullet(B)\otimes_B E^\ad_\l$ by the left quantum Hodge (pseudo) inner product on $\Omega^\bullet(B)$ and the canonical hermitian structure on  $E^\ad_\l$ (\cite{sald1}), $$d^{\nabla^{\omega}_{\ad}\star_\l}$$ is the formal adjoint operator of the exterior derivative of the induced quantum linear connection on $E^\ad_\l$ (just as in the {\it classical} case), $$d^{S^{\omega}\star_\l}$$ is the formal adjoint operator of $\Upsilon_\ad \circ S^{\omega} \circ \Upsilon^{-1}_\ad$ and $$(d^{\nabla^{\omega}_{\ad}\star_\l}-d^{S^{\omega}\star_\l}) R^{\omega}:= (d^{\nabla^{\omega}_{\ad}\star_\l}-d^{S^{\omega}\star_\l}) \Upsilon_{\ad}(R^{\omega}).$$

On the other hand, $\widehat{\Upsilon}_\ad$ is the canonical isomorphism between $$\Mor(\ad,\Delta_\Hor)  \quad \mbox{ and }\quad E^\ad_\r \otimes_B \Omega^\bullet(B),$$
where $E^\ad_\r:=\Mor(\ad,\Delta_P)$ as right $B$--module is the right associated quantum vector bundle of $\ad$. Furthermore, $$\langle-\mid -\rangle_\r$$ is the induced (pseudo) inner product defined on $E^\ad_\r \otimes_B \Omega^\bullet(B)$ by the right quantum Hodge (pseudo) inner product on $\Omega^\bullet(B)$ and the canonical hermitian structure on  $E^\ad_\r$ (\cite{sald1}), $$d^{\nabla^{\omega}_{\ad}\star_\r}$$ is the formal adjoint operator of the exterior derivative of the induced quantum linear connection on $E^\ad_\r$ (just as in the {\it classical} case), $$d^{S^{\omega}\star_\r}$$ is the formal adjoint operator of $\widehat{\Upsilon}_\ad \circ S^{\omega} \circ \widehat{\Upsilon}^{-1}_\ad$,  $$(d^{\widehat{\nabla}^{\omega}_{\ad}\star_\r}-d^{\widehat{S}^{\omega}\star_\r}) \widehat{R}^{\omega}:=(d^{\widehat{\nabla}^{\omega}_{\ad}\star_\r}-d^{\widehat{S}^{\omega}\star_\r})  \widehat{\Upsilon}_{\ad}(\widehat{R}^{\omega})$$ and $$\widehat{\lambda}=\ast \circ \lambda \circ \ast, \qquad \widehat{R}^{\omega}=\ast \circ R^{\omega}\circ \ast.$$ For a much detailed explanation the reader can check  \cite{sald1,sald2}. 

If every qpc in the quantum principal bundle is regular, then $S^\omega$ is always zero and hence $d^{S^{\omega}\star_\l}=d^{S^{\omega}\star_\r}=0$. Even more, if the right structure is equivalent to the left one, then both terms of equation \ref{3.18} are equal and since $\langle-\mid -\rangle$ is a (pseudo) inner product, equation \ref{3.18} becomes (since there is no need to use both structures, we will only use the left structure; so there is no need to keep the subscript $\l$)
\begin{equation}
\label{3.18.1}
(d^{\nabla^{\omega}_{\ad}\star}-d^{S^{\omega}\star}) R^{\omega}=0.
\end{equation} 
 In summary, when every qpc is regular and the right structure is equivalent to the left one, for example, for {\it classical} principal bundles, equation \ref{3.18} turns into the well--know {\it classical} Yang--Mills equation $$d^{\nabla^{\omega}_{\ad}\star}R^{\omega}=0.$$ Of course in general, not every qpc is regular and not always the left structure is equivalent to the right one and this makes the {\it quantum} case so interesting.  For this paper, since the quantum principal bundle that we use is the trivial one and the $\ad$ corepresentation is trivial (see equation \ref{2.15}), all the right structure is equivalent to the left one. In addition, by Proposition \ref{prop1} there is only one regular qpc. Therefore the non--commutative Yang--Mills equation that we have to consider is the equation \ref{3.18.1}

\begin{Proposition}
    \label{prop3}
    The equation \ref{3.18.1} in terms of the non--commutative field strength $F^{\omega}$ is
    \begin{equation}
    \label{3.19}
    (d^{\star}-d^{S^\omega\star})F^\omega(\vartheta)=0,
\end{equation}
where $d^{\star}$ is the quantum codifferential (see equation \ref{2.11}) and (in abuse of notation) $d^{S^\omega\star}$ is (now) the formal adjoint operator of $d^{S^\omega}$ (see equation \ref{3.3}) with respect to $\langle-|-\rangle$. Concretely
\begin{equation}
    \label{3.20}
    d^{S^\omega\star}T: {_{\inv}}\Gamma\longrightarrow \Omega^{k}(B)
\end{equation}
is given by
$$d^{S^\omega\star}T(\vartheta)=(-1)^{k}\,i\,\star^{-1}\left([A^\omega(\vartheta),\star T(\vartheta)]^\partial\right),$$
for all linear maps $T:{_{\inv}}\Gamma\longrightarrow \Omega^{k+1}(B)$, i.e., 
\begin{equation}
    \label{adjoint}
    d^{S^\omega\star}=(-1)^{k+1} \star^{-1} \circ\,\, d^{S^\omega} \circ \star.
\end{equation}
\end{Proposition}

\begin{proof}
    By equations \ref{2.14}, \ref{2.15}, elements $\tau$ $\in$ $\Mor(\ad, \Delta_\H)$ are linear maps $$\tau:{_\inv}\Gamma\longrightarrow \Omega^\bullet(P)$$ such that $\tau(\vartheta)=\eta\otimes \mathbbm{1}$ for some $\eta$ $\in$ $\Omega^\bullet(B)$ and hence 
    \begin{equation*}
    \begin{aligned}
    \Upsilon_\ad: \Mor(\ad, \Delta_\H) &\longrightarrow \Omega^\bullet(B)\otimes_B E^\ad_\l\\
     \tau &\longmapsto \eta\otimes_B T_\mathbbm{1}
    \end{aligned}
    \end{equation*}
    with $T_\mathbbm{1}: {_\inv}\Gamma\longrightarrow P$ the linear map such that $T_\mathbbm{1}(\vartheta)=\mathbbm{1}\otimes \mathbbm{1}$. In particular we have $$\Upsilon_\ad(R^\omega)=F^\omega(\vartheta)\otimes_B T_\mathbbm{1}.$$ 

    On the other hand,  the induced linear connection on $E^\ad_\l$ is given by 
    \begin{equation*}
    \begin{aligned}
    \nabla^\omega_\ad: E^\ad_\l &\longrightarrow \Omega^1(B)\otimes_B E^\ad_\l\\
     T &\longmapsto \Upsilon_\ad \circ D^\omega \circ T
    \end{aligned}
    \end{equation*}
    and its exterior derivative is $d^{\nabla^\omega_\ad}=\Upsilon_\ad \circ D^\omega \circ \Upsilon^{-1}_\ad$. In this way, in accordance with Proposition \ref{cov} we get $$d^{\nabla^\omega_\ad} (\eta \otimes_B T_{\mathbbm{1}})= d\eta\otimes_B T_\mathbbm{1}=dT(\vartheta)\otimes_B T_\mathbbm{1},$$ where $T:{_\inv}\Gamma\longrightarrow \Omega^\bullet(B)$ is the linear map given by $T(\vartheta)=\eta$. This implies that its formal adjoint operator is $$d^{\nabla^\omega_\ad\star} (\eta \otimes_B T_{\mathbbm{1}})=d^\star T(\vartheta)\otimes_B T_\mathbbm{1}.$$ Similarly $$ (\Upsilon_\ad \circ S^\omega \circ \Upsilon^{-1}_\ad)(\eta \otimes_B T_{\mathbbm{1}})= i[A^\omega(\vartheta),\eta]^\partial\otimes_B T_{\mathbbm{1}}=d^{S^\omega}T(\vartheta)\otimes_B T_{\mathbbm{1}}$$ and therefore a large and tedious calculation proves that the formal adjoint operator of $\Upsilon_\ad \circ S^\omega \circ \Upsilon^{-1}_\ad $ is $$d^{S^{\omega}\star}(\eta \otimes_B T_{\mathbbm{1}})= d^{S^\omega \star}T(\vartheta)\otimes_B T_{\mathbbm{1}}$$ where  $$d^{S^\omega\star}T(\vartheta)=(-1)^{k}\,i\,\star^{-1}\left([A^\omega(\vartheta),\star T(\vartheta)]^\partial\right).$$ This shows that equation \ref{3.18.1} is $$0=(d^{\nabla^{\omega}_{\ad}\star}-d^{S^{\omega}\star}) R^{\omega}= (d^{\star}-d^{S^\omega\star})F^\omega(\vartheta)\otimes_B T_{\mathbbm{1}} $$ and in terms of  the non--commutative field strength $F^{\omega}$ we get the equation \ref{3.19}.
\end{proof}

 The reader can compare equation \ref{3.19} with equation \ref{1.3}. It is worth mentioning that as in the {\it classical} case, not every qpc satisfies equation \ref{3.18.1} or equivalently, not every non--commutative gauge potential satisfies equation \ref{3.19}.

\begin{Theorem}
    \label{max2}
    The non--commutative Gauss law and the non--commutative Ampere equation are
    \begin{equation}
    \label{3.21}
    \nabla \cdot {\bf D}=\rho^e,
\end{equation}
\begin{equation}
    \label{3.22}
    \nabla\times {\bf H}-\frac{\partial {\bf D}}{\partial x^0}= {\bf j}^e,
\end{equation}
where
\begin{equation}
    \label{3.23}
    \rho^e=i[{\bf D},{\bf A}^{\ast}],
\end{equation}
and 
\begin{equation}
    \label{3.24}
    {\bf j}^e=i [{\bf D},\phi^{\ast}]-i({\bf H}\times {\bf A}^{\ast}+{\bf A}^{\ast}\times {\bf H}).
\end{equation}
\end{Theorem}

\begin{proof}
    As in the {\it classical} case, a direct calculation in coordinates shows that
    \begin{equation*}
        \begin{aligned}
            d^\star F^\omega(\vartheta)&=-\nabla \cdot {\bf D}\,dx^0+\left(\nabla\times {\bf H}-\frac{\partial {\bf D}}{\partial x^0}\right)_1\,dx^1\\
            &+\left(\nabla\times {\bf H}-\frac{\partial {\bf D}}{\partial x^0}\right)_2\,dx^2+\left(\nabla\times {\bf H}-\frac{\partial {\bf D}}{\partial x^0}\right)_3\,dx^3
        \end{aligned}
    \end{equation*}
    where $\displaystyle \left(\nabla\times {\bf H}-\frac{\partial {\bf D}}{\partial x^0}\right)_i$ denotes the $i$--coordinate of the vector $\displaystyle \nabla\times {\bf H}-\frac{\partial {\bf D}}{\partial x^0}$. While $$d^{S^\omega\star}F^\omega(\vartheta)=-\rho^e\,dx^0+{\bf j}^e_1 \,dx^1+{\bf j}^e_2 \,dx^2+{\bf j}^e_3 \,dx^3$$ where  $$\rho^e=i[{\bf D},{\bf A}^{\ast}]\qquad \mbox{ and }\qquad {\bf j}^e=({\bf j}^e_1,{\bf j}^e_2,{\bf j}^e_3)= i [{\bf D},\phi^{\ast}]-i({\bf H}\times {\bf A}^{\ast}+{\bf A}^{\ast}\times {\bf H})$$ with ${\bf A}^\ast=(A^\ast_1,A^\ast_2,A^\ast_3)$. Now the theorem follows from equation \ref{3.19} and the linear independence of the canonical basis $\{dx^0,dx^1,dx^2,dx^3\}$.
\end{proof}

\begin{Remark}
   \label{remax3}
    Following the {\it classical} interpretation, since the term $\rho^e$ is equal to the divergence of the field ${\bf D}$, we can consider that $\rho^e$ is a (non--commutative) electric charge density. In the same way, since ${\bf j}^e$ is equal to the rotational of the field ${\bf H}$ minus the partial of time of the field ${\bf D}$, we can consider that ${\bf j}^e$ is a (non--commutative) electric current density. It is highly worth remembering that we are considering an electromagnetic field in the vacuum.
\end{Remark}

\begin{Remark}
    \label{remax4}
    As in the classical case, equation \ref{3.19} or equivalently equations \ref{3.21}, \ref{3.22} represents critical points of the non--commutative Yang--Mills functional, the functional that measures the square norm of the curvature of a qpc. In the next subsection we are going to present an explicit $A^\omega$ for which $\rho^e\not=0$.

    Thus, if we ask for $$d^{\star_\l}F^\omega(\vartheta)=0$$ or equivalently $$\nabla \cdot {\bf D}=0,\qquad  \nabla\times {\bf H}-\frac{\partial {\bf D}}{\partial x^0}=0$$ in an attempt to imitate the classical case as in \cite{obs}, $A^\omega$ might not even be a critical point of the non--commutative Yang--Mills functional \cite{sald2}. In other words, $d^{\star_\l}F^\omega(\vartheta)=0$ cannot be the correct dynamical equation of our trivial quantum bundle. The correct one is the equation \ref{3.19} because every critical point of the Yang--Mills functional has to satisfies it. 
\end{Remark}

In summary, equations \ref{3.5},\ref{3.6}, \ref{3.21},\ref{3.22} are the {\it non--commutative Maxwell equations in the vacuum} in the quantum space--time $B$. It is worth remarking again that even in the vacuum there are electric/magnetic charges densities as well as magnetic/electric current densities, which are generated by non--trivial self--interactions of $A^\omega$ and not generating by any other field. 

On the other hand, since the divergence of the curl is zero we can find {\it conservation laws} for all cases:
\begin{equation}
    \label{3.25}
    \nabla\cdot{\bf j}^e+\frac{\partial \rho^e}{\partial x^0}=0, \qquad \nabla\cdot{\bf j}^m+\frac{\partial \rho^m}{\partial x^0}=0.
\end{equation}

The presence of $\rho^m$, ${\bf j}^m$, $\rho^e$, ${\bf j}^e$  in the non--commutative Maxwell equations obtained from the general theory of quantum principal bundles presented in \cite{sald1,sald2} by means of the operator $S^\omega$ is one of the two main goals of this paper.  

\subsection{Some Solutions, Remarks and the Covariant Formulation}

Now we are going to present two easy but non--trivial solutions for our new set of Maxwell equations:  a solution for which the non--commutative potential produces a non--zero magnetic charge density and a solution for which the non--commutative potential produces a non--zero electric charge density and non--zero magnetic current density. 

In fact, by taking $\theta^{\mu\nu}=0$ for all $\mu$, $\nu$ except for $\theta^{23}$; let us consider the non--commutative gauge potential 
\begin{equation}
    \label{3.35}
    A^\omega: {_{\inv}}\Gamma\longrightarrow \Omega^1(B)
\end{equation}
defined by $$A^\omega(\vartheta)=x^3\,dx^2+x^1\cdot x^2\,dx^3.$$  In this way
$${\bf E}=(0,0,0),\quad {\bf B}=(1-x^1,x^2,0),$$
$${\bf D}=(0,0,0),\quad {\bf H}=(1-(1-\theta^{23})x^1,x^2,0).$$ Thus the non--commutative Maxwell equations are satisfied:
$$\nabla \cdot {\bf H}=\theta^{23}, \qquad \nabla \cdot {\bf D}=0,$$ $$ \nabla \times {\bf D}+\frac{\partial {\bf H}}{\partial x^0}=0,\qquad \nabla \times {\bf H}-\frac{\partial {\bf D}}{\partial x^0}=0.$$ It is worth mentioning that the magnetic charge density is the real constant $\theta^{23}$. This solution is consistent with the zero slope limit of string theory:  $\theta^{0j}=0$ for $j=1,2,3$ \cite{string}.

On the other hand, if $\theta^{01}=1$, $\theta^{02}\not=0$, $\theta^{12}=0$ and the non--commutative gauge potential
\begin{equation}
    \label{3.35}
    A^\omega: {_{\inv}}\Gamma\longrightarrow \Omega^1(B)
\end{equation}
is given by $$A^\omega(\vartheta)=x^1\cdot x^2\,dx^0-x^0\,dx^1,$$ we have
$${\bf E}=(-1-x^2,-x^1,0),\quad {\bf B}=(0,0,0),$$
$${\bf D}=(-1+\theta^{02}x^1,-x^1,0),\quad {\bf H}=(0,0,0).$$ Thus the non--commutative Maxwell equations are satisfied:
$$\nabla \cdot {\bf H}=0, \qquad \nabla \cdot {\bf D}=\theta^{02},$$ $$ \nabla \times {\bf D}+\frac{\partial {\bf H}}{\partial x^0}=(0,0,-1),\qquad \nabla \times {\bf H}-\frac{\partial {\bf D}}{\partial x^0}=0.$$ Again, it is worth mentioning that the electric charge density is the real constant $\theta^{02}$ and the magnetic current density is the real vector $(0,0,-1)$. 

\begin{Remark}
\label{rema1}
The general theory on which this paper is based (\cite{sald1}, \cite{sald2}) works with qpcs in the most general way. However, there is a special kind of qpc called {\it real} qpc which satisfies $$\omega(\theta^\ast)=\omega(\theta)^\ast  $$ for all $\theta$ $\in$ ${_\inv}\Gamma$.  This kind of qpc is the only one that other papers consider, for example \cite{micho1,micho2,steve}. In  the {\it classical} case only real principal connections have physical meaning; so perhaps in the {\it non--commutative} case only real qpc's have physical meaning as well. 

For our quantum bundle, this condition becomes into $$A^\omega(\theta^\ast)=A^\omega(\theta)^\ast$$ for all $\theta$ $\in$ ${_\inv}\Gamma$.
It is worth mentioning that the explicit solutions presented above come from real qpc's and we have obtained real constants and vectors as charges and currents, what in principle could be measurable quantities. Of course, there are more solutions, these were only two examples.
\end{Remark}

Now we will pass to the tensor index notation with the metric $\eta_{\mu\nu}=\mathrm{diag}(1,-1,-1,-1)$.  Let us consider
\begin{equation}
    \label{3.38}
    A_\mu=(\phi,-{\bf A})
\end{equation}
and then equation \ref{2.39} is given by
\begin{equation}
   \label{3.39}
    F_{\mu\nu}=\partial_\mu A_\nu-\partial_\nu A_\mu-i[A_\mu,A_\nu].
\end{equation}

The geometrical (topological) equation, the equation \ref{3.4}, can be written as  
\begin{equation}
    \label{3.40}
    \partial_\mu \widetilde{F}^{\mu\nu}=J^{m\,\nu} \qquad \mbox{ where }\qquad J^{m\,\nu}=(\rho^m,{\bf -j}^m)=i[A_\mu,\widetilde{F}^{\mu\nu}]=i[A_\mu,\widetilde{F}^{\mu\nu}_{\mathrm{classical}}].
\end{equation}
Here, we are consider that
\begin{equation}
    \label{3.41}
    \widetilde{F}^{\mu\nu}={1\over 2}\epsilon^{\mu\nu\alpha\beta}F_{\alpha\beta},
\end{equation}
is the dual electromagnetic tensor field with $\epsilon^{\mu\nu\alpha\beta}$ the Levi--Civitta symbol and $\widetilde{F}^{\mu\nu}_{\mathrm{classical}}$ is the {\it classical} part of $\widetilde{F}^{\mu\nu}$, i.e., the part of $\widetilde{F}^{\mu\nu}$ only with  terms of ${\bf E}$ and ${\bf B}$. The reader can compare equation \ref{3.40} with its {\it classical} counterpart, the equation $A.14$.

On the other hand, the dynamical equation, the equation \ref{3.19}, can be written as
\begin{equation}
    \label{3.42}
    \partial_\mu F^{\mu\nu}=J^{e\,\nu} \qquad \mbox{ where } \quad F^{\mu\nu}=\eta^{\mu\alpha}\eta^{\nu\beta}F_{\mu\nu} \quad  \mbox{ and } \quad J^{e\,\nu}=(\rho^e,{\bf j}^e)=i[A^\ast_\mu,F^{\mu\nu}]
\end{equation}
with $A^\ast_\mu=(\phi^\ast,-{\bf A}^\ast)$. The reader can compare the last equation with its {\it classical} counterpart, equation $A.15$.  

Notice that the form of both four--currents are like four--currents in {\it classical} non--abelian gauge theories. Conservation laws are given by
\begin{equation}
    \label{3.46}
    \partial_\mu J^{e\,\mu}=0,\qquad \partial_\mu J^{m\,\mu}=0;
\end{equation} 
while solutions in terms of the four--potential in the Lorentz gauge 
\begin{equation}
    \label{3.44}
    \partial_\mu A^{\mu}=0
\end{equation}
are given by
\begin{equation}
    \label{3.45}
    \partial_\mu\partial^\mu A^{\nu}=i[A_\mu,F^{\mu\nu}]+i[A^{\mu},\partial_\mu A^{\nu}].
\end{equation}

\begin{Remark}
    \label{rema2}
    In equation \ref{3.39}, the term $$i[A_\mu,A_\nu] $$ comes from the general definition of the curvature for our quantum bundle (\cite{micho2,steve}) with this specific differential calculus. In other papers, as in  \cite{elec,twisted,ruiz}, this term comes from the non--commutativity of $B$ in the definition of the so called gauge curvature $$i[D_\mu,D_\nu], \qquad \mbox{ where }\qquad D_\mu=\partial_\mu-iA_\mu.$$  However, the change $$\partial_\mu \,\longrightarrow \, D^\omega_\mu=\partial_\mu- i\, A_\mu,$$ has no sense for our development. Since we are in the vacuum, there is no other other field coupled with the electromagnetic field, i.e., there is no reason to use any other associated vector bundle than the one associated to $\ad:{_\inv}\Gamma \longrightarrow {_\inv}\Gamma \otimes H$  (see equation \ref{2.15}) and in accordance with Proposition \ref{cov} $$D^\omega_\mu=\partial_\mu$$ in this associated bundle, as in the classical case (see Remark \ref{rema}).

    If in our quantum bundle we consider the classical algebra of differential forms of $U(1)$, we would get $$F_{\mu\nu}=\partial_\mu A_\nu-\partial_\nu A_\mu$$ because in this case, the only embedded differential is $\Theta=0$, regardless of the quantum base space.  Furthermore, in this situation the geometrical (topological) equation and the dynamical equation of the quantum trivial bundle would be $$dF^\omega=0,\qquad d^{\star_\l}F=0,$$ as in the classical case. Thus, in order to recreate the {\it common} form of the curvature used in the literature (\cite{elec,twisted,ruiz}) we decided to use the $\ast$-FODC $(\Gamma,d)$ given in equation \ref{2.13}. However and just like we have seen in the previous subsections, this turns the geometrical (topological) and the dynamical equations into equations \ref{3.2} and \ref{3.19}.
\end{Remark}    

In the next section we will use another differential calculus for $H$ and as a result we will get two gauge fields, as well as new  geometrical (topological) and dynamical equations.

\begin{Remark}
    \label{rema3.1}
    In the same way, it is worth mentioning that the four--current terms are consequence of the presence of the operator $S^\omega$ in the non--commutative Bianchi identity and in the non--commutative Yang--Mills equation. They do not come from $D^\omega$ as in the classical case for non--abelian theories (see Proposition \ref{cov} and Remark \ref{rema}) . 
\end{Remark}

\begin{Remark}
   \label{rema4}
     In addition to the Remarks \ref{remax2}, \ref{remax4}, another reason to consider that the mathematical formulation presented in this paper (which comes from the general theory showed in \cite{sald2}) is the correct one is the fact that we can obtain equations \ref{3.2} and \ref{3.19} just by performing the correct calculations on $\Omega^\bullet(B)$. Indeed, in order to get the equation \ref{3.2} it is enough to calculate $$dF^\omega$$ and see that it is not zero anymore, it is exactly $d^{S^\omega}F^\omega$.  
     
     On the other hand, let us take the standard electromagnetic action for any non--commutative gauge potential (real or not)
     \begin{equation}
    \label{3.43}
\qS_\YM:=\int_B F^{\mu\nu}\cdot F^{\ast}_{\mu\nu}\; \dvol= \int_B \langle F^{\omega}(\vartheta), F^{\omega}(\vartheta)\rangle \; \dvol=\langle F^{\omega}(\vartheta)\mid F^{\omega}(\vartheta)\rangle.
     \end{equation}
    Then, by considering  $A^\omega \mapsto  A^\omega +z\lambda$ with $z$ $\in$ $\C$ and $\lambda:{_\inv}\Gamma \longrightarrow \Omega^1(M)$ a linear map, a simple and direct calculation shows that $$ \left.\dfrac{\partial}{\partial z}\right|_{z=0}\qS_\YM(A^\omega +z\lambda)=0$$ for all $\lambda$ if and only if $$ (d^{\star_\l}-d^{S^\omega\star_\l})F^\omega(\vartheta)=0.$$  It is worth mentioning that the last action is exactly the one proposed in the \cite{sald2} for this particular quantum bundle. For a real qpc the last action becomes $$\qS_\YM:=\int_B F^{\mu\nu}\cdot F_{\mu\nu}\; \dvol.$$ 
\end{Remark}

In accordance with \cite{twisted}, twisted gauge transformations are symmetries of the last action. Furthermore, the covariant formulation allows to identify that our equations are  covariant under the correct set of Lorentz transformations $\Lambda$, for example, the ones who leads the matrix $(\theta^{\mu\nu})$ invariant  \cite{lorentz}.

\begin{Remark}
    \label{rema4}
    The whole classical electromagnetic field theory (Appendix $A$) can be recovered by considering
    $$\theta^{\mu\nu}\longrightarrow 0,$$ even with the non--standard $1D$ differential calculus of $U(1)$ because by the graded commutativity
 of $\Omega^\bullet(B)$ we would have $F^\omega=dA^\omega$ (see equation \ref{2.36}).
\end{Remark}

Finally, it is worth mentioning that there is a subjective asymmetry in the model: the electric four--current depends on the specific form of the four--potential; while the magnetic four--current does not. This four--current exists only by the presence of the electromagnetic field in the quantum space--time $B$; it does not depend on the form of the four--potential (this is because the non--commutative Bianchi identity is satisfied for every qpc). This asymmetry is resolved in the model presented in the next section.

\section{Two Electromagnetic Gauge Fields}

The purpose of this section is to present a model of the electromagnetic interaction for which there are instantons that are not solutions of the corresponding Yang--Mills equation. 

\subsection{The Quantum Bundle}

For this section we will work with the quantum principal $U(1)$--bundle presented in \ref{2.29}. However, we are going to chance the differential calculus on the bundle.  

Indeed, let us consider $(\Omega^\bullet(B),d,\ast)$ but with a Euclidean metric. For example
$$\displaystyle \langle \sum^3_{\mu=0}f_\mu\,dx^\mu ,\sum^3_{\nu=0}h_\nu\,dx^\nu\rangle^1_\l= \sum^3_{\mu,\nu=0}\delta^{\mu\,\nu} f_\mu\cdot h^\ast_\nu=f_0\cdot h^\ast_0+f_1\cdot h^\ast_1+f_2\cdot h^\ast_2+f_3\cdot h^\ast_3.$$ Of course, this change implies 
\begin{equation}
    \label{b.1}
    \star_\l\mathbbm{1}=\dvol, \;\;\star_\l \dvol=\mathbbm{1},
\end{equation}
\begin{equation}
    \label{b.2}
    \begin{aligned}
    \star_\l dx^0=dx^1\wedge dx^2\wedge dx^3, \quad \star_\l dx^1=-dx^0\wedge dx^2\wedge dx^3,\\
    \star_\l dx^2=dx^0\wedge dx^1\wedge dx^3,\quad \star_\l dx^3=-dx^0\wedge dx^1\wedge dx^2,
    \end{aligned}
\end{equation}
\begin{equation}
    \label{b.3}
    \begin{aligned}
    \star_\l dx^0 \wedge dx^1=dx^2\wedge dx^3, \quad \star_\l dx^0 \wedge dx^2=-dx^1\wedge dx^3, \quad \star_\l dx^0 \wedge dx^3=dx^1\wedge dx^2,\\
    \star_\l dx^1 \wedge dx^2=dx^0\wedge dx^3, \quad \star_\l dx^1 \wedge dx^3=-dx^0\wedge dx^2, \quad \star_\l dx^2 \wedge dx^3=dx^0\wedge dx^1,
    \end{aligned}
\end{equation}
\begin{equation}
    \label{b.4}
    \begin{aligned}
    \star_\l dx^1 \wedge dx^2\wedge dx^3 =-dx^0, \quad \star_\l dx^0 \wedge dx^2 \wedge dx^3=dx^1,\\
    \star_\l dx^0 \wedge dx^1 \wedge dx^3=-dx^2, \quad \star_\l dx^0 \wedge dx^1 \wedge dx^2=dx^3
    \end{aligned}
\end{equation}
and we have $\star^2_\l=(-1)^{k(4-k)}\,\id.$ 

Now, for the new differential calculus on $H$ let us take $q$ $\in$ $\R-\{0,1\}$, $\mathcal{L}=\mathrm{span}_\C\{z,z^{-1}\}$ and its linear dual space $\hat{\mathcal{L}}:=\mathrm{span}_\C\{\theta_-,\theta_+\}$, where $$\theta_-(z)=1,\quad \theta_-(z^{-1})=0,\quad \theta_+(z)=0,\quad \theta_+(z^{-1})=1.$$ The map
\begin{equation*}
\varpi: \Ker(\epsilon)\longrightarrow \hat{\mathcal{L}},\qquad
g\longmapsto \varpi(g),
\end{equation*}
where $\varpi(g)(x)={\mathcal{Q}}(g\otimes x)$ with ${\mathcal{Q}}$ such that ${\mathcal{Q}}(z^m\otimes z^n)=q^{2mn}$ for all $m$, $n$ $\in$ $\Z$; defines 
\begin{equation}
    \label{b.5}
    (\Gamma,d),
\end{equation}
a $\ast$--FODC by means of its space of invariant elements, or equivalently, its (dual) quantum Lie algebra (\cite{libro,woro,steve})
\begin{equation}
    \label{b.6}
    {_{\inv}}\Gamma:=\frac{\Ker(\epsilon)}{\Ker(\varpi)}.
\end{equation}
It is worth remarking that $\mathrm{dim}( {_{\inv}}\Gamma)=2$  and clearly a linear basis of the quantum Lie algebra is given by
\begin{equation}
    \label{b.7}
    \beta:=\{ \vartheta^e:=-i\,\theta_-,\,\vartheta^m:=-i\,\theta_+\}.
\end{equation}
By considering the quantum germs map \cite{steve,woro}
\begin{equation}
    \label{b.8}
    \pi:\Ker(\epsilon)\longrightarrow {_{\inv}}\Gamma
\end{equation}
it is easy to look for relations, such as
\begin{equation}
    \label{b.9}
    \begin{aligned}
    \pi(z^n)&=i(q^{2n}-1)\vartheta^e+i(q^{-2n}-1)\vartheta^m,\\
    \pi(z^{-n})&=i(q^{-2n}-1)\vartheta^e+i(q^{2n}-1)\vartheta^m.
    \end{aligned}
\end{equation}
Also we can calculate the adjoint coaction of $G$ on ${_{\inv}}\Gamma$
\begin{equation}
    \label{b.10}
    \begin{aligned}
\ad: {_{\inv}}\Gamma &\longrightarrow {_{\inv}}\Gamma\otimes G\\
\vartheta &\longmapsto \vartheta\otimes \mathbbm{1}
\end{aligned}
\end{equation}
because of $$\ad(\pi(g))=((\pi\otimes \id)\circ \Ad)(g)=\pi(g)\otimes \mathbbm{1}$$ for all $g$ $\in$ $G$.

Now we have to take the universal differential envelope $\ast$--calculus \cite{micho1,micho2,steve}
\begin{equation}
    \label{b.11}
    (\Gamma^\wedge,d,\ast).
\end{equation}

Notice
\begin{equation}
    \label{b.12}
    \vartheta^{e\ast}=\vartheta^e,\quad \vartheta^{m\ast}=\vartheta^m 
\end{equation}
\begin{equation}
    \label{b.13}
    d\pi(z)=d\pi(z^{-1})=-\frac{(q^2-1)^2}{q^2}(\vartheta^e\vartheta^m+\vartheta^m\vartheta^e),
\end{equation}
\begin{equation}
    \label{b.14}
d\vartheta^e=d\vartheta^m=i(\vartheta^e\vartheta^m+\vartheta^m\vartheta^e).
\end{equation}
The algebra $(\Gamma^\wedge,d,\ast)$ will play the role of {\it quantum} differential forms of $U(1)$ in this section. 

Now let us take the the trivial differential calculus $$ \Omega^\bullet(P):=\Omega^\bullet(B)\otimes \Gamma^\wedge, \;\;\Delta_{\Omega^\bullet(P)}:=\id\otimes \Delta:\Omega^\bullet(P)\longrightarrow \Omega^\bullet(P)\otimes \Gamma^\wedge.$$ For this case, by equations \ref{b.12}, \ref{b.13}, \ref{b.14} we can take the embedded differential 
\begin{equation}
    \label{b.15}
    \Theta: {_{\inv}}\Gamma\longrightarrow {_{\inv}}\Gamma\otimes {_{\inv}}\Gamma
\end{equation}
given by $$\Theta(\vartheta^e)=\Theta(\vartheta^m)= i(\vartheta^e\otimes \vartheta^m+\vartheta^m\otimes\vartheta^e).$$

Following the {\it classical} case, let us consider the non--commutative gauge potential $$A^\omega:{_\inv}\Gamma\longrightarrow \Omega^1(B)$$ given by   
\begin{equation}
\label{b.16}
A^\omega(\vartheta^e)= \phi^e\, dx^0-A^e_1\, dx^1-A^e_2\, dx^2-A^e_3\, dx^3,
\end{equation}
\begin{equation}
\label{b.17}
A^\omega(\vartheta^m)= \phi^m dx^0\,-A^m_1\, dx^1-A^m_2\, dx^2-A^m_3\, dx^3.
\end{equation}
Thus the curvature $$F^\omega=dA^\omega-\langle A^\omega,A^\omega\rangle:{_\inv}\Gamma\longrightarrow \Omega^2(B)$$ in the elements of the basis $\beta$ is
\begin{equation}
    \label{b.18}
    F^\omega(\vartheta^e)=dA^\omega(\vartheta^e)-i A^\omega(\vartheta^e)\wedge A^\omega(\vartheta^m)-i A^\omega(\vartheta^m)\wedge A^\omega(\vartheta^e),
\end{equation}
\begin{equation}
    \label{b.19}
    F^\omega(\vartheta^m)=dA^\omega(\vartheta^m)-i A^\omega(\vartheta^e)\wedge A^\omega(\vartheta^m)-i A^\omega(\vartheta^m)\wedge A^\omega(\vartheta^e).
\end{equation}
In terms of coordinates we get
\begin{equation}
    \label{b.20}
    F^\omega(\vartheta^e)=\sum_{0\leq \mu < \nu \leq 3}F^e_{\mu\nu}\, dx^\mu\wedge dx^\nu\;\; \mathrm{ where }\;\; (F^e_{\mu\nu})=\begin{pmatrix}
0 & D^e_1 & D^e_2 & D^e_3  \\
-D^e_1 & 0 & -H^e_3 & H^e_2 \\
-D^e_2 & H^e_3 & 0 & -H^e_1\\
-D^e_3 & -H^e_2 & H^e_1 & 0
\end{pmatrix},
\end{equation}
\begin{equation}
    \label{b.21}
    {\bf D}^e=(D^e_1,D^e_2,D^e_3):={\bf E}^e+i[\phi^m,{\bf A}^e]+i[\phi^e,{\bf A}^m],\;\quad {\bf E}^e:= -\frac{\partial{\bf A}^e}{\partial x^0}-\nabla \phi^e
\end{equation}
and 
\begin{equation}
    \label{b.22}
    {\bf H}^e=(H^e_1,H^e_2,H^e_3):={\bf B}^e+i {\bf A}^e\times {\bf A}^m + i{\bf A}^m \times {\bf A}^e, \;\quad {\bf B}^e:= \nabla \times {\bf A}^e,
\end{equation}
where ${\bf A}^e=(A^e_1,A^e_2,A^e_3)$ and ${\bf A}^m=(A^m_1,A^m_2,A^m_3)$; while 
\begin{equation}
    \label{b.23}
    F^\omega(\vartheta^m)=\sum_{0\leq \mu < \nu \leq 3}F^m_{\mu\nu}\, dx^\mu\wedge dx^\nu\;\; \mathrm{ where }\;\; (F^m_{\mu\nu})=\begin{pmatrix}
0 & H^m_1 & H^m_2 & H^m_3  \\
-H^m_1 & 0 & -D^m_3 & D^m_2 \\
-H^m_2 & D^m_3 & 0 & -D^m_1\\
-H^m_3 & -D^+_2 & D^m_1 & 0
\end{pmatrix},
\end{equation}
\begin{equation}
    \label{b.24}
    {\bf H}^m=(H^m_1,H^m_2,H^m_3):={\bf B}^m+i[\phi^m,{\bf A}^e]+i[\phi^e,{\bf A}^m],\;\quad {\bf B}^m:= -\frac{\partial{\bf A}^m}{\partial x^0}-\nabla \phi^m
\end{equation}
and 
\begin{equation}
    \label{b.25}
    {\bf D}^m=(D^m_1,D^m_2,D^m_3):={\bf E}^m+i{\bf A}^e\times {\bf A}^m +i {\bf A}^m \times {\bf A}^e, \;\quad {\bf E}^m:= \nabla \times {\bf A}^m.
\end{equation}
The notation chosen is not a coincidence: we will consider that $A^\omega(\vartheta^e)$ is  {\it the non--commutative electric potential $1$--form}; $(F^e_{\mu,\nu}),$ ${\bf D}^e$ and ${\bf H}^e$ are {\it the non--commutative electromagnetic tensor field, the non--commutative electric field and the non--commutative magnetic field} generated by  $A^\omega(\vartheta^e)$, respectively;  and ${\bf E}^e$ and ${\bf B}^e$ are their corresponding {\it classical} parts. In the same way, we will consider that $A^\omega(\vartheta^m)$ is {\it the non--commutative magnetic potential $1$--form}; $(F^m_{\mu,\nu}),$ ${\bf D}^m$ and ${\bf H}^m$ are {\it the non--commutative magnetoelectric tensor field, the non--commutative electric field and the non--commutative magnetic field} generated by  $A^\omega(\vartheta^m)$, respectively; and ${\bf E}^m$ and ${\bf B}^m$ are their corresponding {\it classical} parts. 

\subsection{Non--Commutative Maxwell equations and Instantons}

For this quantum bundle and following the same argumentation of the proof of Proposition \ref{prop1}, the non--commutative Bianchi identity (equation \ref{3.1}) becomes  
\begin{equation}
    \label{b.26}
    (d-d^{S^\omega})F^\omega=\langle A^\omega, \langle A^\omega,A^\omega\rangle\rangle- \langle \langle  A^\omega, A^\omega\rangle,A^\omega \rangle,
\end{equation}
where $d$ is the differential and now the operator
\begin{equation}
    \label{b.27}
    d^{S^\omega}T:{_\inv}\Gamma\longrightarrow \Omega^{k+1}(B)
\end{equation}
is given by
$$ d^{S^\omega}T(\vartheta^e)=d^{S^\omega}T(\vartheta^m)=i\,[A^\omega(\vartheta^e),T(\vartheta^m)]^\partial+i\,[A^\omega(\vartheta^m),T(\vartheta^e)]^\partial\;\;\;\in\;\;\; \Omega^{k+1}(B)$$ for all linear maps $T: {_{\inv}}\Gamma\longrightarrow \Omega^k(B)$.

In the same way, the equation \ref{3.18} becomes 
\begin{equation}
    \label{b.28}
    (d^{\star}-d^{S^\omega\star})F^\omega(\vartheta^e)=(d^{\star}-d^{S^\omega\star})F^\omega(\vartheta^m)=0,
\end{equation}
 where $d^{\star}$ is the quantum codifferential and $d^{S^\omega\star}$ is the formal adjoint operator of $d^{S^\omega}$, which is now given by
 $$d^{S^\omega\star}T(\vartheta^e)=(-1)^{k}\,i\,\star^{-1}\left([A^\omega(\vartheta^e)+A^\omega(\vartheta^m),\star T(\vartheta^m)]^\partial\right) \;\in\; \Omega^k(B) ,$$ $$d^{S^\omega\star}T(\vartheta^m)=(-1)^{k}\,i\,\star^{-1}\left([A^\omega(\vartheta^e)+A^\omega(\vartheta^m),\star T(\vartheta^e)]^\partial\right) \;\in\; \Omega^k(B),$$ for all linear maps $T:{_{\inv}}\Gamma\longrightarrow \Omega^{k+1}(B)$.

It is worth mentioning that 
\begin{equation}
    \label{b.29}
    d^{S^\omega\star}\not=(-1)^{k+1} \star^{-1} \circ\, d^{S^\omega} \circ \star.
\end{equation}

Direct calculations as in Propositions \ref{max1}, \ref{max2} show that the  {\it Euclidean non--commutative Maxwell equations in the vacuum} for  $A^\omega(\vartheta^e)$ are
\begin{equation}
    \label{b.30}
    \nabla\cdot {\bf H}^e=\rho
\end{equation}

\begin{equation}
    \label{b.31}
    \nabla\times {\bf D}^e+\frac{\partial {\bf H}^e}{\partial x^0}=-{\bf j}.
\end{equation}

\begin{equation}
    \label{b.32}
    \nabla \cdot {\bf D}^e=\rho^e
\end{equation}

\begin{equation}
    \label{b.33}
    \nabla\times {\bf H}^e+\frac{\partial {\bf D}^e}{\partial x^0}= -{\bf j}^e;
\end{equation}
while the {\it Euclidean non--commutative Maxwell equations vacuum} for $A^\omega(\vartheta^m)$ are
\begin{equation}
    \label{b.34}
   \nabla\cdot {\bf D}^m=\rho
\end{equation}

\begin{equation}
    \label{b.35}
    \nabla\times {\bf H}^m+\frac{\partial {\bf D}^m}{\partial x^0}=-{\bf j}
\end{equation}

\begin{equation}
    \label{b.36}
    \nabla \cdot {\bf H}^m=\rho^m
\end{equation}

\begin{equation}
    \label{b.37}
   \nabla\times {\bf D}^m+\frac{\partial {\bf H}^m}{\partial x^0}=-{\bf j}^m,
\end{equation}
where 
\begin{equation}
    \label{b.38}
    \rho=i[{\bf B}^e,{\bf A}^m]+i[{\bf E}^m,{\bf A}^e]
\end{equation}
is the magnetic charge density generated by $A^\omega(\vartheta^e)$ and it also the electric charge density generated by $A^\omega(\vartheta^m)$, 
\begin{equation}
    \label{b.39}
    -{\bf j}=i[\phi^m,{\bf B}^e]+i[\phi^e,{\bf E}^m]-i({\bf E}^e\times{\bf A}^m+{\bf A}^m\times{\bf E}^e+{\bf B}^m\times{\bf A}^e+{\bf A}^e\times{\bf B}^m)
\end{equation}
is the magnetic current density generated by $A^\omega(\vartheta^e)$ and it is also the electric current density generated by $A^\omega(\vartheta^m)$,
\begin{equation}
    \label{b.40}
    \rho^e=i[{\bf A}^{m\ast},{\bf D}^e+{\bf H}^m]\quad \mbox{ and }\quad \rho^m=i[{\bf A}^{e\ast},{\bf D}^e+{\bf H}^m],
\end{equation}
are the electric and magnetic charge density generated by $A^\omega(\vartheta^e)$ and $A^\omega(\vartheta^m)$ respectively, and
\begin{equation}
    \label{b.41}
    -{\bf j}^e=i[{\bf D}^e+{\bf H}^m,\phi^{m\ast}]+i(({\bf H}^e+{\bf D}^m)\times {\bf A}^{m\ast}+{\bf A}^{m\ast}\times ({\bf H}^e+{\bf D}^m)),
\end{equation}
\begin{equation}
    \label{b.42}
    -{\bf j}^m=i[{\bf D}^e+{\bf H}^m,\phi^{e\ast}]+i(({\bf H}^e+{\bf D}^m)\times {\bf A}^{e\ast}+{\bf A}^{e\ast}\times ({\bf H}^e+{\bf D}^m))
\end{equation}
are the electric current density and the magnetic current density generated by $A^\omega(\vartheta^e)$ and $A^\omega(\vartheta^m)$, respectively. As in the previous section, these charges and currents satisfy conservation laws
\begin{equation}
    \label{b.43}
    \nabla\cdot{\bf j}^e+\frac{\partial \rho^e}{\partial x^0}=0, \qquad \nabla\cdot{\bf j}^m+\frac{\partial \rho^m}{\partial x^0}=0,\qquad \nabla\cdot{\bf j}+\frac{\partial \rho}{\partial x^0}=0.
\end{equation}

In summary, for this model critical points of the non--commutative Yang--Mills functional, i.e., solutions of the non--commutative Yang--Mills equation (equation \ref{b.28})  are given by non--commutative gauge potentials $$A^\omega:{_\inv}\Gamma\longrightarrow \Omega^1(B)$$ for which $A^\omega(\vartheta^e)$ satisfies equations \ref{b.32} \ref{b.33} and $A^\omega(\vartheta^m)$ satisfies equations \ref{b.36} \ref{b.37}. Since equations \ref{b.30}, \ref{b.31} and equations \ref{b.34}, \ref{b.35} come from the non--commutative Bianchi identity (equation \ref{b.26}), these equations are satisfy for every $A^\omega$.

It is worth mentioning that Proposition \ref{cov} holds for this model (see equation \ref{b.10}) and hence, the four--current terms come from the operator $S^\omega$, not from $D^\omega$ as in the {\it classical} case for non--abelian theories.

In accordance with \cite{sald2}, we get the Yang--Mills action of the system as 
\begin{equation}
    \label{b.45}
    \qS_\YM=\qS_\YM^e+\qS_\YM^m,
\end{equation}
where
\begin{equation}
    \label{b.46}
    \qS_\YM^e=\int_B F^{e\;\mu\nu}\cdot F^{e\,\ast}_{\mu\nu}\qquad \mbox{ and }\qquad \qS_\YM^m=\int_B F^{m\;\mu\nu}\cdot F^{m\,\ast}_{\mu\nu}.
\end{equation}

\begin{Remark}
    As in the model showed in Section $3$, equations \ref{b.26} and \ref{b.28} have been found by applying the general theory presented in \cite{sald2} on this specific quantum bundle. However they can be found by calculating $dF^\omega$ and by performing the variation on $\qS_\YM$ $$A^\omega \mapsto A^\omega +z \lambda $$  with $z$ $\in$ $\C$ and $\lambda:{_\inv}\Gamma \longrightarrow \Omega^1(M)$ a linear map  and asking for $$ \left.\dfrac{\partial}{\partial z}\right|_{z=0}\qS_\YM(A^\omega +z\lambda)=0$$ for all $\lambda$. Therefore, we conclude that equations \ref{b.26} and \ref{b.28} are the correct geometrical (topological) and dynamical equations for this model.
\end{Remark}

As in the previous section, twisted gauge transformations (\cite{twisted}) and Lorentz transformations that leads the matrix $(\theta^{\mu\nu})$ invariant (\cite{lorentz}) are symmetries of the last action. Furthermore, the whole model is invariant under the following transformation 
\begin{equation}
    \label{b.44}
    A^\omega(\vartheta^e)\;\longleftrightarrow  A^\omega(\vartheta^m).
\end{equation}
In other words 
$${\bf D}^e\longrightarrow {\bf H}^m,\;\; {\bf E}^e\longrightarrow {\bf B}^m, \;\; \qquad {\bf H}^m\longrightarrow {\bf D}^e, \;\; {\bf B}^m\longrightarrow {\bf E}^e $$
$${\bf H}^e\longrightarrow {\bf D}^m, \;\; {\bf B}^e\longrightarrow {\bf E}^m, \;\;\qquad {\bf D}^m\longrightarrow {\bf H}^e, \;\; {\bf E}^m\longrightarrow {\bf B}^e, $$
$$\rho^e \longrightarrow \rho^m,\;\; {\bf J}^e \longrightarrow {\bf J}^m,\;\;\qquad \;\; \rho^m \longrightarrow \rho^e,\;\; {\bf J}^m \longrightarrow {\bf J}^e. $$  This symmetry produces a kind of {\it non--commutative} dual transformations and it comes from the interchanging of the order basis $\beta=\{\vartheta^e,\vartheta^m\}\,\longleftrightarrow\, -\beta=\{\vartheta^m,\vartheta^e\}$; so it is a $\Z_2$--symmetry. In other words, there is full symmetry between the electric field and the magnetic field in non--commutative Maxwell equations. 

The whole magnetic field of the model is given by ${\bf H}:={\bf H}^e+{\bf H}^m$ and the whole electric field of the model is ${\bf D}:={\bf D}^e+{\bf D}^m$ and now the subjective asymmetry described at the end of the previous sections is resolved: both four--currents have a part that depend on the form of the potential and a part that is present only by the existence of the fields in $B$.

The following proposition is the second main goal of this paper; however, by the form of equations \ref{b.26}, \ref{b.28} and \ref{b.29} it should be not surprising.

\begin{Proposition}
    \label{instantons}
    There are instantons that are no solutions of the non--commutative Yang--Mills equation. In our context, an instanton is a non--commutative gauge potential $$A^\omega:{_\inv}\Gamma\longrightarrow \Omega^1(B) $$ for which $$\star_\l F^\omega=-F^\omega.$$ Since $\beta=\{\vartheta^e, \vartheta^m\}$ is basis of ${_\inv}\Gamma$, $A^\omega$ is an instanton if and only if $\star_\l F^\omega(\vartheta^e)=-F^\omega(\vartheta^e)$ and $\star_\l F^\omega(\vartheta^m)=-F^\omega(\vartheta^m)$. Equivalently ${\bf D}^e={\bf H}^e$ and ${\bf D}^m={\bf H}^m$.
\end{Proposition}

\begin{proof}
    Let us consider $\theta^{\mu\nu}\not=0$ for some $\mu,\nu=1,2,3$ and the non--commutative gauge potential $A^\omega$ where $A^\omega(\vartheta^e)$ is given by (see equation \ref{b.16}) $$\phi^e=-2(x^2)^2+(y^2)^2+(z^2)^2,\qquad A^e=(4x^2\cdot x^3,2\,x^1\cdot x^3,6\,x^1\cdot x^2)$$ and $A^\omega(\vartheta^m)$ is given by (see equation \ref{b.17}) $$\phi^m=0,\qquad A^m=(0,0,0).$$ Then $${\bf D}^e={\bf H}^e=(4\,x^1,-2\,x^2,-2\,x^3) \quad \mbox{ and }\quad {\bf H}^m={\bf D}^m=(0,0,0);$$ so $A^\omega$ is an instanton. Now a direct calculation shows that $A^\omega(\vartheta^e)$ satisfies equations \ref{b.30}, \ref{b.31}, \ref{b.32}, \ref{b.33}; however,  $A^\omega(\vartheta^m)$ does not satisfy equation \ref{b.36} and therefore $A^\omega$ is not a solution of the non--commutative Yang--Mills equation.
\end{proof}

Of course, there are instantons that are solutions of the non--commutative Yang--Mills equation, for example, the ones for which $F^\omega=0$.

The model presented in this section is only a {\it mathematical model} in the sense that it does not represent the physical word because there are not two gauge fields associated to the electromagnetic interaction.  As we have mentioned in the introduction of this paper, the using of two electromagnetic gauge fields comes from Cabiboo--Ferrari's idea (\cite{cf,cv}). However, by using the theory of quantum principal bundles we can perform the model with only one $U(1)$ as the gauge group. The importance of this model lies in the full symmetry between the electric field and the magnetic field (equation \ref{b.44}) and Proposition \ref{instantons}. In non--commutative geometry, there are instantons that are not solutions of the corresponding Yang--Mills equation. This opens the possibility to find a model in non--commutative geometry bundle with better description of the physical word for which Proposition \ref{instantons}  holds. 

\section{Concluding Comments}

In Differential Geometry the most general framework of the electromagnetic field theory is given by the theory of principal bundles and principal connections \cite{na}. So it should be natural to think that in the non--commutative case the correct mathematical framework of the electromagnetic field theory is given by the theory of quantum principal bundles and quantum principal connections, and showing this and all its consequences is exactly the purpose of this paper.

As the reader should had already notice throughout the entire text, specially by Remark \ref{rema2}, the general theory of quantum principal bundles reproduces previous results of the non--commutative $U(1)$--gauge theory as well as clarify the correct equation of motion and the correct geometrical (topological) equation. This implies that in order to formulate a non--commutative gauge theory we need to proceed in the most general way, it is not enough with taking the well--known Lagrangians and change the product because that explicit form of the Lagrangians and the equation of motion were found in the {\it classical} context, there is not guarantee that they will work in the non--commutative context, specially because of the operator $S^\omega$. This paper is proof of that.

The operator $S^\omega$ is not zero only in the non--commutative context and for that reason it changes completely the mathematical formulation of these kind of models. It is worth mentioning that this operator appears naturally in the non--commutative Bianchi identity (\cite{micho2,sald2}) and it also appears naturally when you varying the qpc in the non--commutative Yang--Mills functional (\cite{sald2}), so it is essential to consider it in non--commutative gauge theory. In the particular case of the quantum principal bundle used in Sections 2 and 3, the operator $S^\omega$ turns the electromagnetic field into a kind of dyon gauge field in the sense that it produces electric and magnetic charges and currents by itself in the vacuum (see Remark \ref{rema3.1}).

Notice that in the Euclidean case, instantons are still Yang--Mills connections. Indeed, if $A^{\omega}$ is an instanton, by equation \ref{3.2} and the fact that 
\begin{equation}
    \label{dual}
    d^{S^{\omega}\star_\l}=(-1)^{k+1} \star^{-1}_\l \circ \;  d^{S^{\omega}} \circ \star_\l;
\end{equation}
we get $$(d^{\star_\l}-d^{S^{\omega}\star_\l})F^{\omega}= (\star^{-1}_\l \circ \;  (d-d^{S^{\omega}}) \circ \star_\l)F^{\omega}=-\star^{-1}_\l(d-d^{S^\omega})F^{\omega}=0.$$ Instanton has been studied in detail in the non--commutative context, for example in \cite{ins}. Nevertheless,  equation \ref{dual} does not hold in general as in the {\it mathematical model} presented in the Section 4. This automatically implies that instantons as solutions of the corresponding Yang--Mills equation will depend on the specific form of  equations \ref{3.1}, \ref{3.18} in the quantum principal bundle used. All of these are reasons to keep the research going. In particular, it is possible to add scalar matter fields (unfortunately, the general theory presented in \cite{sald2} do not take into account spin matter) to the theory presented in this paper or to develop a theory for the rest of the Lie groups of the standard model. Moreover, the general formulation present in \cite{sald1,sald2} allows to work with other quantum space--time spaces like the $\kappa$--Minkowski space--time \cite{obs}. This space together with the Moya--Weyl algebra are the most studied spaces for quantum gravity in the framework of non--commutative geometry.

\begin{appendix}
\section{Classical Maxwell equations}
Consider a principal $U(1)$--bundle over the Minkowski space--time: $\R^4$ with the metric $\eta=\mathrm{diag}(1,-1,-1,-1)$.

Since $\R^4$ is contractible, every bundle over it is trivializable, so without lose of generality, let us consider the trivial principal bundle
\begin{equation}
\label{1.0}
  \mathrm{proj}: \R^4 \times U(1)\longrightarrow \R^4.  
\end{equation}

 In this case, principal connections $$\omega: T(\R^4\times U(1))\longrightarrow  \mathfrak{u}(1)$$ are in bijection with globally defined $\mathfrak{u}(1)$--valued differential $1$--forms: $i\,A^\omega$, where $i=\sqrt{-1}$ and $$A^\omega=\displaystyle \sum^3_{\mu=0} A_{\mu}\,dx^{\mu}\; \in \; \Omega^1(\R^4)$$  is usually called {\it the potential 1--form}; and the vector  $(A_0,A_1,A_2,A_3)$ created by the coefficients of $A^\omega$ receives the name of {\it four--vector potential}. For this case the curvature form   $$R^\omega: \Omega^2(\R^4\times U(1))\longrightarrow  \mathfrak{u}(1)$$ is related with the potential $1$--form by means of  the (de Rham) differential, i.e., curvature forms are in bijection with 
\begin{equation}
    \label{1.1}
    dA^\omega=:F^\omega \; \in \; \Omega^2(\R^4).
\end{equation}

In general every principal connection has to satisfies the (second) Bianchi identity; however due to the fact that $U(1)$ is an {\it abelian} group, for our case this identity reduces to (\cite{na})
\begin{equation}
    \label{1.2}
    0=dF^\omega=d^2 A^\omega,
\end{equation}
which is a trivial relation taking into account the properties of $d$. Nevertheless it is worth remarking again that the last equation arises from: the relation between the curvature and the $1$--form potential, the Bianchi identity and the commutative of $U(1)$; all of this in the context of the graded--commutative de Rham algebra $\Omega^\bullet(\R^4)$.  

On the other hand, the Yang--Mills functional is a functional from the space of principal connections to $\R$ which measures the squared norm of the curvature of a principal connection and the Yang--Mills equation comes from a variational principle in which we are looking for critical points. For our case and again, by the relation between the curvature and the $1$--form potential and the commutative of $U(1)$, the Yang--Mills equation is (\cite{na})
\begin{equation}
    \label{1.3}
    0=d^\star F^\omega=d^\star dA^\omega,
\end{equation}
where $d^\star=(-1)^k \star^{-1}\circ\, d \,\circ \star$ is the codifferential, the formally adjoint operator of $d$ and $\star$ is the star Hodge operator \cite{na}. 

Now from equations \ref{1.2}, \ref{1.3} it is possible to obtain Maxwell equations \cite{na}. In fact, by choosing $(A_0,A_1,A_2,A_3)=(\phi,-{\bf A})$ with ${\bf A}$ a vector field in $\R^3$, the equation \ref{1.1} becomes into
\begin{equation}
    \label{1.4}
    F^\omega=\sum_{0\leq \mu < \nu \leq 3}F_{\mu\nu}\, dx^\mu\wedge dx^\nu\quad \mathrm{ where }\quad (F_{\mu\nu})=\begin{pmatrix}
0 & E_1 & E_2 & E_3  \\
-E_1 & 0 & -B_3 & B_2 \\
-E_2 & B_3 & 0 & -B_1\\
-E_3 & -B_2 & B_1 & 0
\end{pmatrix},
\end{equation}
where 
\begin{equation}
    \label{1.5}
    {\bf E}=(E_1,E_2,E_3):=-\frac{\partial{\bf A}}{\partial x^0}-\nabla \phi
\end{equation}
is the {\it electric field},
\begin{equation}
    \label{1.6}
    {\bf B}=(B_1,B_2,B_3):=\nabla \times \bf A
\end{equation}
is the {\it magnetic field} and the $2$--form $(F_{\mu,\nu})$ is called {\it the electromagnetic field tensor}. By substituting the value of $F^\omega$ in equation \ref{1.2} we get the Gauss Law for magnetism and the Faraday equation
\begin{equation}
    \label{1.7}
    \nabla\cdot {\bf B}=0,
\end{equation}
\begin{equation}
    \label{1.8}
    \nabla\times {\bf E}+\frac{\partial {\bf B}}{\partial x^0}=0;
\end{equation}
while by substituting it in equation \ref{1.3} we get the Gauss Law and the Ampere equation
\begin{equation}
    \label{1.9}
    \nabla\cdot {\bf E}=0
\end{equation}
\begin{equation}
    \label{1.10}
    \nabla\times {\bf B}-\frac{\partial {\bf E}}{\partial x^0}=0.
\end{equation}
Equations \ref{1.7}--\ref{1.10} are the Maxwell equations in the vacuum \cite{na}. $\phi$ receives the name of {\it the electric potential} or {\it scalar potential}; while ${\bf A}$ receives the name of {\it the magnetic potential} or {\it vector potential}.

In tensor index notation we have
\begin{equation}
    A_\mu=(\phi,-{\bf A})
\end{equation}
and
\begin{equation}
    F_{\mu\nu}=\partial_\mu A_\nu-\partial_\nu A_\mu.
\end{equation}
The geometrical equation (equation \ref{1.2}) can be written as  
\begin{equation}
    \label{1.11}
    \partial_\mu \widetilde{F}^{\mu\nu}=0,
\end{equation}
and the dynamical equation (equation \ref{1.3}) can be written as
\begin{equation}
    \label{1.12}
    \partial_\mu F^{\mu\nu}=0,
\end{equation}
where $$F^{\mu\nu}=\eta^{\mu\alpha}\eta^{\nu\beta}F_{\mu\nu},\qquad \widetilde{F}^{\mu\nu}={1\over 2}\epsilon^{\mu\nu\alpha\beta}F_{\alpha\beta},$$ with $\epsilon^{\mu\nu\alpha\beta}$ the Levi--Civita symbol and $\eta_{\mu\nu}=\eta^{\mu\nu}= \mathrm{diag}(1,-1,-1,-1)$. The tensor $\widetilde{F}_{\mu\nu}$ receives the name of the {dual electromagnetic field tensor}.

It is worth mentioning that equation \ref{1.1} shows explicitly the gauge transformation symmetry: 
\begin{equation}
    \label{1.13}
    A^\omega\longrightarrow {A^\omega}':=A^\omega+d\chi \quad \mbox{ or in tensor index notation } \quad A'_\mu:=A_\mu+\partial_\mu\chi,
\end{equation}
for any $f$ $\in$ $C^\infty(\R^4)$.

\end{appendix}

\end{document}